\providecommand{\drfi}{draft}
\documentclass[oneside,article,\drfi]{memoir}
\usepackage{amsmath}         
\usepackage{amsthm,thmtools,thm-restate} 
\usepackage{enumitem}        

\usepackage{tikz}
  \usetikzlibrary{matrix,arrows,positioning}
  \usetikzlibrary{backgrounds}
  \usepackage{tikz-cd}
\usepackage{float} 
\usepackage{caption, subcaption} 

\usepackage[nobysame,alphabetic,initials]{amsrefs} 
\usepackage[hidelinks]{hyperref}                   
\usepackage[capitalize,nosort]{cleveref}           

\usepackage{amssymb}
\usepackage[mathscr]{euscript}
\usepackage{relsize}
\usepackage{stmaryrd}
\usepackage{stackengine}
\usepackage{tipa}

\usepackage{indentfirst} 
\usepackage{multicol}
\usepackage{multirow}

\usepackage[inline,nomargin]{fixme} 
  \fxsetup{targetlayout=color}

\usepackage{float} 

\tikzstyle{vertex}=[circle, inner sep=0pt, minimum size=4pt, line width=8pt]

\isopage[12]

\setlrmargins{*}{*}{1}
\checkandfixthelayout

\counterwithout{section}{chapter}
\usepackage{appendix}

\newcommand{\C}{\mathcal{C}}
\newcommand{\D}{\mathcal{D}}

\newcommand{\ModR}{\mathsf{Mod}_R}
\newcommand{\ModRN}{\mathsf{Mod}_R^{\mathbb{N}}}
\newcommand{\Ch}{\mathsf{Ch}} 
\newcommand{\Top}{\mathsf{Top}}

\newcommand{\Set}{\mathsf{Set}}

\newcommand{\DiGraph}{\mathsf{DiGraph}}

\newcommand{\Cof}{\mathsf{DiGraph}_2}
\newcommand{\CofPO}{\mathsf{DiGraph}_2^{\mathsf{PO}}}

\newcommand{\Gen}{\mathsf{Gen}}
\newcommand{\W}{(X - A)_1}

\newcommand{\tOmega}{\widehat{\Omega}}
\newcommand{\tA}{\widehat{A}}
\newcommand{\tOmegaOne}{\widehat{\Omega}^1}
\newcommand{\tAOne}{\widehat{A}^1}

\newcommand{\pushout}{\arrow [dr, phantom, "\ulcorner" very near end]}
\newcommand{\pullback}{\arrow [dr, phantom, "\lrcorner" very near start]}

\newcommand{\id}[1][]{\mathrm{id}_{#1}}
\newcommand{\gtimes}{\mathbin{\square}}
\newcommand{\bd}{\partial}
\newcommand{\op}{\mathsf{op}}

\newcommand{\inj}{\mathsf{inj}}
\newcommand{\proj}{\mathsf{proj}}

  \newcommand{\weto}{\xrightarrow{\sim}}
  
  \newcommand{\cto}{\rightarrowtail}

  \newcommand{\into}{\hookrightarrow}

\newcommand{\ie}{i.e.,}
\newcommand{\from}{\colon}

\newtheorem{theorem}{Theorem}[section] 
\newtheorem{corollary}[theorem]{Corollary}
\newtheorem{lemma}[theorem]{Lemma}
\newtheorem{proposition}[theorem]{Proposition}

\theoremstyle{definition}
\newtheorem{definition}[theorem]{Definition}

\theoremstyle{remark}
\newtheorem{remark}[theorem]{Remark}
\newtheorem{example}[theorem]{Example}

\declaretheorem[style=plain,numbered=no,name=Theorem]{theorem*}

\usepackage[
textsize=tiny,
colorinlistoftodos]
{todonotes} 

\title{Cofibration category of digraphs for path homology}
\author{Daniel Carranza \and Brandon Doherty \and Krzysztof Kapulkin \and Morgan Opie \and Maru Sarazola \and Liang Ze Wong}
\date{\today}

\begin{document}

\maketitle

\begin{abstract}
    We prove that the category of directed graphs and graph maps carries a cofibration category structure in which the weak equivalences are the graph maps inducing isomorphisms on path homology.
\end{abstract}


\section*{Introduction}

Homology theories, among other homotopy-theoretic invariants, play an important role in graph theory.
Examples of such homology theories of graphs (cf.~\cite[\S2.4]{barcelo-greene-jarraj-welker:discrete-path}) include the clique homology, which is the homology of the clique complex associated to a graph; CW-homology, i.e., the homology of the graph viewed as a $1$-dimensional CW-complex; and cubical homology, which is the homology of the $1$-coskeletal cubical set associated to a graph \cite{barcelo-capraro-white}.
Path homology, introduced by Grigor'yan, Lin, Muranov, and Yau \cite{grigor'yan-lin-muranov-yau:unpublished}, is yet another such invariant, however it is fundamentally an invariant of \emph{directed} graphs, or digraphs.
It is most closely related to magnitude homology \cite{hepworth-willerton}, as shown recently by Asao \cite{asao:magnitude-n-path}.

Path homology has seen significant development over the last 10 years.
On the foundational side, this includes the development of the corresponding homotopy theory \cite{grigor'yan-lin-muranov-yau:homotopy}, which in turn allows for the statement and proof of the Eilenberg--Steenrod axioms \cite{grigor'yan-jimenez-muranov-yau}.
On the computational side, a K{\"u}nneth-style theorem was proven in \cite{grigor'yan-muranov-yau}.
Finally, these techniques found applications both within mathematics, e.g., a new proof of the classical Sperner Lemma \cite[\S5]{grigor'yan-lin-muranov-yau:homotopy}, and outside, e.g., in directed network analysis \cite{chowdhury-memoli,chowdhury}.

Since its introduction, path homology has been vastly generalized.
First, from digraphs to path complexes \cite{grigor'yan-lin-muranov-yau}, which are combinatorial objects similar to, yet more general than, simplicial complexes.
In particular, both digraphs and simplicial complexes are canonically examples of path complexes and path homology of path complexes specializes both to path homology of digraphs and to simplicial homology of simplicial complexes.
The second generalization \cite{ivanov-pavutniskiy} was to the category of path sets, a presheaf category similar to that of simplicial sets.

The goal of the present paper is to investigate path homology using tools from abstract homotopy theory, in particular, the framework of cofibration categories.
Our main theorem is:

\begin{theorem*}[cf.~\cref{cofib-cat}]
  The category of directed graphs carries a cofibration category structure in which the weak equivalences are the graph maps inducing isomorphisms on path homology groups.
\end{theorem*}

(Co)fibration categories were developed by Brown \cite{brown:abstract-homotopy-theory} under the name `categories of fibrant objects,' as a framework for studying generalized cohomology theories, but have since found many other applications, e.g., in formal logic \cite{avigad-kapulkin-lumsdaine}.
Cofibration categories are a slight weakening of a more common notion of a model category, as developed by Quillen \cite{quillen:book}.
More precisely, a cofibration category structure on a category $\C$ consists of two classes of morphisms in $\C$: cofibrations and weak equivalences, subject to some axioms making it possible to speak of and conveniently work with homotopy colimits in $\C$ \cite{szumilo:two-models,kapulkin-szumilo}.

To define our cofibration category structure, we build on the development of path homology, especially in papers \cite{grigor'yan-jimenez-muranov-yau} and \cite{grigor'yan-muranov-yau}.
In particular, our definition of cofibration (\cref{cofib-def}) is a strengthening of the `no-outgoing-edges condition' used in \cite[\S5]{grigor'yan-jimenez-muranov-yau}.
However, as explained in \cref{rmk:cofib-both-conditions}, this would not be sufficient and hence we require the existence of a projecting decomposition, introduced in \cite{leinster:magnitude-graph}.

Our work provides additional insight into the structural properties of path homology.
For instance, the fact that the class of cofibrations in \cref{cofib-def} is not the saturation of a small set (\cref{no-gen-cof-set}) suggests that additional axioms will be required, as indicated in \cite[Rmk.~5.3]{grigor'yan-jimenez-muranov-yau}, to uniquely determine path homology.

\textbf{Related work.}
While we are unaware of similar work in the category of digraphs, considerations similar to ours are not without precedent in the category of (undirected) graphs.
In \cite{carranza-kapulkin:cubical-setting}, a fibration category structure is constructed on the category of simple graphs in which the weak equivalences are the weak homotopy equivalences of discrete homotopy theory \cite{babson-barcelo-longueville-laubenbacher}.
On the other hand, in \cite{goyal-santhanam}, it is proven that no model category structure exists in the category of undirected graphs with loops in which weak equivalences are the $\times$-homotopy equivalences of Dochtermann \cite{dochtermann} and cofibrations are a subclass of monomorphisms.

\textbf{Organization.}
This paper is structured as follows.
We begin by recalling the necessary notions related to digraphs, path homology, and cofibration categories in \cref{sec:preliminaries}.
In \cref{sec:cofibs}, we introduce our notion of cofibration of digraphs and study their basic properties.
The technical heart of the paper is contained in \cref{sec:excision}, where we prove the excision property, i.e., that relative homology induces isomorphisms on homotopy pushouts.
Finally, we assemble all of these results together in \cref{sec:main-thm}, proving our main theorem.

\textbf{Acknowledgements.}
This work began during a Research-in-Teams event at the University of Western Ontario, Canada in May 2022, in which all six authors participated.
The event was supported by awards from Western's Faculty of Science and the Fields Institute in Toronto, Ontario.

This material is based upon work supported by the National Science Foundation under Grant No.~DMS-1928930 while D.C., B.D., K.K., and M.S.~participated in a program supported by the Mathematical Sciences Research Institute.
The program was held in the summer of 2022 in partnership with the Universidad Nacional Autónoma de México.

It is also based upon work supported by the National Science Foundation
under Grant No.~DMS-1928930 while D.C.~and K.K.~participated in a program hosted by the Mathematical Sciences Research Institute in Berkeley, California, during the Fall 2022 semester.

While working on this paper, B.D.~was supported by a grant from the Knut and Alice Wallenberg Foundation, entitled  ``Type Theory for Mathematics and Computer Science'' (principal investigator: Thierry Coquand). M.O.~was supported by the National Science Foundation under Award No.~2202914.

We are very grateful for this support.

\section{Preliminaries} \label{sec:preliminaries}

In this section, we review and establish the necessary background for the results of \cref{sec:cofibs,sec:excision,sec:main-thm}.
We begin by defining (the category of) digraphs and establishing a few facts about colimits therein.
We then review the notion of path homology of a digraph, following \cite{grigor'yan-lin-muranov-yau:unpublished,grigor'yan-lin-muranov-yau:homotopy,grigor'yan-lin-muranov-yau}, computing path homology of a few small graphs and referencing our Python script \cite{python_script} for computations of larger examples.
Finally, we review the requisite background on cofibration categories in preparation for our main theorem asserting the existence of a cofibration category of digraphs.

\subsection*{The category of digraphs}



We begin by defining the category of directed graphs, or digraphs, as a reflective subcategory of a presheaf category.

\begin{definition}
     Define the category $\mathbb{G}$ to be generated by the graph
 \[   \begin{tikzcd}
    V\rar["s",shift left=2] \rar["t"',shift right=2] & E\lar["r" description]
    \end{tikzcd}\]
    subject to the relations $rs=rt=\id$.
\end{definition}

\begin{definition}
    The category of \emph{directed multigraphs} $\Set^{\mathbb{G}^\op}$ is the category of functors $\mathbb{G}^\op \to \Set$.
\end{definition}
Explicitly, a directed multigraph $X$ consists of a set $X_V$ of vertices and a set $X_E$ of edges, together with functions (which we denote with a slight abuse of notation):
\[ s,t\colon X_E\to X_V \quad r \from X_V \to X_E \] 
where $s$ and $t$ pick out the source and target vertices of each (directed) edge and $r\colon X_V\to X_E$ sends a vertex to a ``degenerate'' self-edge.
A morphism $f \from X \to Y$ of directed multigraphs is a natural transformation, i.e.~a pair of functions $(f_V \from X_V \to Y_V, f_E \from X_E \to Y_E)$ which preserve sources, targets, and ``degenerate'' self-edges.
That is, for $e \in X_E$ and $v \in X_V$,
\[ s(f_E(e)) = f_V(s(e)) \quad t(f_E(e)) = f_V(t(e)) \quad r(f_V(v)) = f_E(r(v)).  \]

We denote an edge of a directed multigraph $e\in X_E$ with source $v \in X_V$ and target $w \in X_V$ by $v \to w$.
\begin{definition}
    A \emph{directed graph} (or \emph{digraph}) is a directed multigraph $X\colon\mathbb{G}^\op\to\Set$ such that the function $$X_E\xrightarrow{(s,t)} X_V\times X_V$$ is injective; i.e. there is at most one edge $v\to w$ for any pair of vertices $(v,w)$.
\end{definition}
Let $\DiGraph$ denote the full subcategory of $\Set^{\mathbb{G}^\op}$ spanned by digraphs.
For a morphism $f \from X \to Y$ between digraphs, the function $f_E \from X_E \to Y_E$ is uniquely determined by $f_V \from X_V \to Y_V$.
Thus, the data of a digraph map consists of a function $X_V \to Y_V$ between vertices such that if $v \to w$ in $X$ then either $f(v) = f(w)$ (as every vertex has a self-edge) or $f(v) \to f(w)$ in $Y$.


\begin{remark}
    One may equivalently define a digraph as a set with a reflexive binary relation and a digraph map as a function which preserves this relation.
\end{remark}

To fix the notation for specific digraphs used later in the paper, we now discuss several examples of digraphs.

\begin{example}
    The empty digraph $\varnothing$ is given by the functor $X\colon\mathbb{G}^\op\to\Set$ with $X_V=X_E=\varnothing$. This is an initial object in $\DiGraph$.  
\end{example}

\begin{definition}\label{def:I_n}
    For each $n\geq 0$, the digraph $I_n$ has vertices $0,1,\dots, n$, and a unique edge $i\to i+1$ for each $0\leq i<n$. It can be depicted as

    \begin{figure}[H]
    \centering
    \begin{tikzpicture}[node distance=20pt]
        \node(0) {$\bullet$};
        \node(1) [right=of 0] {$\bullet$};
        \node(2) [right=of 1] {$\bullet$};
        \node[right=of 2] (3)  {\dots};
        \node(4) [right=of 3] {$\bullet$};

        \node[above=1pt of 0] {0};
        \node[above=1pt of 1] {1};
        \node[above=1pt of 2] {2};
        \node[above=1pt of 4] {$n$};

        \draw[->] (0) to (1);
        \draw[->] (1) to (2);
        \draw[->] (2) to (3);
        \draw[->] (3) to (4);
    \end{tikzpicture}
    \end{figure}
\end{definition}

In particular, the graph $I_0$ consists of a single vertex; it is a terminal object in $\DiGraph$.
Note that maps $I_0 \to X$ are in a one-to-one correspondence with vertices of $X$, i.e., the functor $(-)_V \colon \DiGraph \to \Set$ taking a digraph to its set of vertices is representable, represented by $I_0$.

Similarly, the functor $(-)_E \colon \DiGraph \to \Set$ taking a digraph to its set of edges is representable, represented by $I_1$.
Indeed, maps $I_1 \to X$ correspond bijectively to edges of $X$ (including degenerate edges).

\begin{definition}\label{def:Cn}
    For each $n\geq 0$, the oriented $n$-cycle $C_n$ is the digraph with vertices $0,1,\dots, n-1$, edges $i\to i+1$ for each $0 \leq i<n-1$, and an edge $n-1 \to 0$. For example, the cycle $C_8$ may be depicted as 
    
    
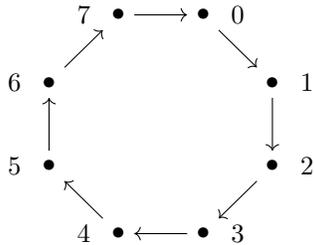
\begin{figure}[H]
        \centering
        \begin{tikzpicture}[node distance=20pt]
            \node(0) {$\bullet$};
            \node(1) [below right=of 0] {$\bullet$};
            \node(2) [below=of 1] {$\bullet$};
            \node(3) [below left=of 2] {$\bullet$};
            \node(4) [left=of 3] {$\bullet$};
            \node(5) [above left=of 4] {$\bullet$};
            \node(6) [above=of 5] {$\bullet$};
            \node(7) [left=of 0] {$\bullet$};
            
            \node[right=1pt of 0] {$0$};
            \node[right=1pt of 1] {$1$};
            \node[right=1pt of 2] {$2$};
            \node[right=1pt of 3] {$3$};        
            \node[left=1pt of 4] {$4$};
            \node[left=1pt of 5] {$5$};
            \node[left=1pt of 6] {$6$};
            \node[left=1pt of 7] {$7$};
            
            \draw[->] (0) to (1);
            \draw[->] (1) to (2);
            \draw[->] (2) to (3);
            \draw[->] (3) to (4);
            \draw[->] (4) to (5);
            \draw[->] (5) to (6);
            \draw[->] (6) to (7);
            \draw[->] (7) to (0);
        \end{tikzpicture}
    \end{figure}
\end{definition}
\begin{definition}\label{def:alternating_cycle}
    For each $k\geq 1$, the \emph{alternating $n$-cycle} $\tilde C_{2k}$ has vertices $0,1,\dots, 2k-1$ and edges $i\to i+1$ for each $0 \leq i\leq 2k-1$ even, and $i+1 \to i$ for each $0 \leq i \leq 2k-1$ odd    (where vertices are taken modulo $2k$ as needed). For example:
        
        \begin{figure}[H]
        \centering
          \begin{subfigure}[b]{0.4\textwidth}
          \centering
            \begin{tikzpicture}[node distance=20pt]
                \node(0) {$\bullet$};
                \node(1) [right=of 0] {$\bullet$};
                \node(2) [below=of 1] {$\bullet$};
                \node(3) [left=of 2] {$\bullet$};
                
                \node[left=1pt of 0] {$0$};
                \node[right=1pt of 1] {$1$};
                \node[right=1pt of 2] {$2$};   
                \node[left=1pt of 3] {$3$};   
                
                \draw[->] (0) to (1);
                \draw[->] (2) to (1);
                \draw[->] (2) to (3);
                \draw[->] (0) to (3);
            \end{tikzpicture}
              \caption{$\tilde C_4$}
            \end{subfigure}
             \begin{subfigure}[b]{0.3\textwidth}
            \centering
            \begin{tikzpicture}[node distance=20pt]
                \node(0) {$\bullet$};
                \node(1) [above right=of 0] {$\bullet$};
                \node(2) [right=of 1] {$\bullet$};
                \node(3) [below right = of 2]{$\bullet$};
                \node(4) [below left = of 3]{$\bullet$};
                \node(5) [left = of 4]{$\bullet$};
                
                \node[left=1pt of 0] {$0$};
                \node[left=1pt of 1] {$1$};
                \node[right=1pt of 2] {$2$};  
                \node[right=1pt of 3] {$3$};
                \node[right=1pt of 4] {$4$};
                \node[left=1pt of 5] {$5$};  
                
                \draw[->] (0) to (1);
                \draw[->] (2) to (1);
                \draw[->] (2) to (3);
                \draw[->] (4) to (3);
                \draw[->] (4) to (5);
                \draw[->] (0) to (5);
            \end{tikzpicture}
            \caption{$\tilde C_6$}
        \end{subfigure}
        \caption{Depictions of alternating cycles}
    \end{figure}
\end{definition}
\begin{definition}\label{def:Cmn}
    For each $m,n\geq 0$, the $(m,n)$-cycle $C_{m,n}$ is the digraph with vertices $0,1,\dots, m+n-1$, edges $i\to i+1$ for each $0 \leq i < m$, edges $i+1 \to i$ for each $m \leq i < m+n-1$,  and an edge $0 \to m+n-1$. For example, the cycles $C_{2,1},$ $C_{3,1}$ and $C_{3,2}$ may be depicted:
    \begin{figure}[H]
        \centering
        \begin{subfigure}[b]{0.3\textwidth}
          \centering
            \begin{tikzpicture}[node distance=20pt]
                \node(0) {$\bullet$};
                \node(1) [above right=of 0] {$\bullet$};
                \node(2) [below right=of 1] {$\bullet$};
                
                \node[left=1pt of 0] {$0$};
                \node[above=1pt of 1] {$1$};
                \node[right=1pt of 2] {$2$};   
                
                \draw[->] (0) to (1);
                \draw[->] (0) to (2);
                \draw[->] (1) to (2);
            \end{tikzpicture}
              \caption{$C_{2,1}$}
              \label{tilde_C3}
            \end{subfigure}
        \begin{subfigure}[b]{0.3\textwidth}
            \centering
            \begin{tikzpicture}[node distance=20pt]
                \node(0) {$\bullet$};
                \node(1) [right=of 0] {$\bullet$};
                \node(2) [below=of 1] {$\bullet$};
                \node(3) [below=of 0] {$\bullet$};
                
                \node[left=1pt of 0] {$0$};
                \node[right=1pt of 1] {$1$};
                \node[right=1pt of 2] {$2$};   
                \node[left=1pt of 3] {$3$};   
                
                \draw[->] (0) to (1);
                \draw[->] (1) to (2);
                \draw[->] (2) to (3);
                \draw[->] (0) to (3);
            \end{tikzpicture}
            \caption{$C_{3,1}$}
            \label{fig:3-1square}
        \end{subfigure}          
        \begin{subfigure}[b]{0.3\textwidth}
            \centering
            \begin{tikzpicture}[node distance=20pt]
                \node(0) {$\bullet$};
                \node(1) [below right=of 0] {$\bullet$};
                \node(2) [below=of 1] {$\bullet$};
                \node(4) [below left=of 0] {$\bullet$};                
                \node(3) [below=of 4] {$\bullet$};
                
                \node[left=1pt of 0] {$0$};
                \node[right=1pt of 1] {$1$};
                \node[right=1pt of 2] {$2$};   
                \node[left=1pt of 3] {$3$};   
                \node[left=1pt of 4] {$4$};   
                
                \draw[->] (0) to (1);
                \draw[->] (1) to (2);
                \draw[->] (2) to (3);
                \draw[->] (4) to (3);
                \draw[->] (0) to (4);
            \end{tikzpicture} 
            \caption{$C_{3,2}$}
            \label{fig:3-2cycle}
        \end{subfigure}    
        \caption{Depictions of $C_{m,n}$ cycles}
    \end{figure}   
\end{definition}
The inclusion $\DiGraph \into \Set^{\mathbb{G}^\op}$ admits a left adjoint taking a directed multigraph to the digraph obtained by collapsing all parallel edges.
It follows that the category of digraphs admits all small limits and colimits.
Moreover, the limits are computed separately on vertices and edges, while the colimits are first computed in the category $\Set^{\mathbb{G}^\op}$ of directed multigraphs and then reflected using the left adjoint.

The following results provide a convenient characterization of pushouts in $\DiGraph$ of induced subgraph inclusions.

\begin{lemma}\label{pushout-char}
Let $A \hookrightarrow X$ denote an inclusion of directed graphs, with $A$ an induced subgraph of $X$, and let $B \hookrightarrow Y$ denote its pushout along a map $f \colon A \to B$, as depicted below.
\[
\begin{tikzcd}
A \arrow[r,"f",above] \arrow[d,hook] \pushout & B \arrow[d,hook] \\
X \arrow[r] & Y \\
\end{tikzcd}
\]
Then the vertex set of the pushout object $Y$ is given by $Y_V = (X_V \setminus A_V) \sqcup B_V$. 
Edges of $Y$ are determined as follows.
\begin{itemize}
    \item For  $x, y \in X_V \setminus A_V$, there is an edge $x \to y$ in $Y$ if and only if there is an edge $x \to y$ in $X$.
    \item For $a, b \in B_V$, there is an edge $a \to b$ in $Y$ if and only if there is an edge $a \to b$ in $B$.
    \item For $x \in X_V \setminus A_V$ and $a \in B_V$, there is an edge $x \to a$ if and only if there is an edge $x \to \overline{a}$ for some $\overline{a} \in A_V$ with $f(\overline{a}) = a$. 
    Similarly, there is an edge $a \to x$ if and only if there is an edge $\overline{a} \to x$ for some $\overline{a} \in A_V$ with $f(\overline{a}) = a$.
\end{itemize}
The map $X \to Y$ acts as the identity on the complement $X \setminus A$ and restricts to $f$ on $A$. \qed
\end{lemma}

\begin{corollary}\label{pushout-complement}
In the situation of \cref{pushout-char}, the map $X \to Y$ restricts to an isomorphism of complements $X \setminus A \cong Y \setminus B$. 
Moreover, a vertex $y \in Y_V \setminus B_V$ admits a path to some vertex of $B$ if and only if the vertex $x \in X \setminus A$ corresponding to $y$ under the above isomorphism admits a path to some vertex of $A$. In this case, the minimum length of a path in $Y$ from $y$ to a vertex of $B$ is equal to the minimum length of a path in $X$ from $x$ to a vertex of $A$. \qed
\end{corollary}
We will also use the following construction:
\begin{definition}
    The box product of two graphs $X$ and $Y$ is the graph $X \gtimes Y$ with vertices $X_V \times Y_V$
    and an edge $(x,x') \to (y,y')$ when either of the following conditions holds:
    \begin{itemize}
        \item there is an edge $x \to y$ in $X$ and $x' = y'$, or
        \item there is an edge $x' \to y'$ in $Y$ and $x = y$.
    \end{itemize}
\end{definition}

\subsection*{Path homology}

We now define path homology of digraphs, following \cite{grigor'yan-lin-muranov-yau:unpublished,grigor'yan-lin-muranov-yau:homotopy}. 

\begin{definition} \label{def:chain-all-paths}
Let $X$ be a digraph, and let $R$ be a commutative ring.
Define the following $R$-modules
\begin{align*}
    K_n (X ; R) &= R \{X_V^{n+1}\}, \\
    DK_n (X ; R) &= R \{ (x_0, x_1, \dots, x_n ) \in X_V^{n+1} \,|\, x_i = x_{i+1} \text{ for some } 0 \leq i < n \}, \\
    C_n (X ; R) & = K_n X / DK_n X.
\end{align*}
\end{definition}

A generator of $K_n (X ; R)$ can be thought of as a path of length $n$ in the complete digraph on the vertices of $X$.
The generators of the submodule $DK_n (X ; R)$ can be thought of as \emph{degenerate} paths of length $n$ that contain self-loops $x \to x$, which are referred to as \emph{non-regular} paths in \cite[\S2.3]{grigor'yan-lin-muranov-yau:unpublished}.
The quotient $C_n (X ; R)$ is then generated by the regular paths, which are paths that do not contain self-loops.

There are differentials on $K_n (X ; R)$ given by the usual alternating sum formula:
\[
    \bd_n(v_0, \ldots, v_n)=\sum_{i=0}^n (-1)^i(v_0, \ldots, \hat{v_{i}}, \ldots, v_n).
\]
These differentials satisfy $\bd^2 = 0$ (i.e. $\bd_{n-1} \bd_n = 0$).
Further, $\bd_n$ restricts to a map $DK_n (X ; R) \to DK_{n-1} (X ; R)$, and thus passes to the quotient, resulting in a well-defined map:

\begin{definition}
    Let $\bd_n \colon C_n (X ; R) \to C_{n-1} (X ; R)$ be the map induced by $\bd_n \colon K_n (X ; R) \to K_{n-1} (X ; R)$.
\end{definition}

\begin{lemma}[{\cite[Lemma 2.4]{grigor'yan-lin-muranov-yau:unpublished}}]
    With the definitions above, $(C_\bullet (X ; R), \bd_\bullet)$ is a chain complex. \qed
\end{lemma}


\begin{definition}\label{defn:A_n} 
Let $X$ be a digraph, and let $R$ be a commutative ring.
Define the $R$-modules:
\begin{align*}
    \widetilde{A}_n (X ; R) &= R(\DiGraph(I_n, X))\text{,} \\
    D\widetilde{A}_n (X ; R) &= R \left\{f \colon I_n \to X \ | \ f \text{ factors through a map } I_n \to I_{n-1}\right\}\text{,} \\
    A_n (X ; R) &= \widetilde{A}_n (X ; R) / D\widetilde{A}_n (X ; R)\text{.}
\end{align*}
\end{definition}

In the nomenclature of \cite{grigor'yan-lin-muranov-yau:unpublished}, the elements of $\DiGraph(I_n, X)$ are called \emph{allowed paths}, and consequently $\widetilde{A}_n (X ; R)$ is the \emph{$R$-module of allowed paths}.
The quotient $A_n (X ; R)$ is then the \emph{$R$-module of allowed regular paths}.




Since $A_n (X ; R)$ is a submodule of $C_n (X ; R)$, we may restrict $\bd_n \colon C_n (X ; R) \to C_{n-1} (X ; R)$ to $A_n (X ; R)$ obtaining $\bd_n \colon A_n (X ; R) \to C_{n-1} (X ; R)$.
However, the restricted map need not have image in $A_{n-1} (X; R)$ (e.g., consider $X = I_2$). Hence, $A_\bullet (X ; R)$ does not naturally form a chain complex.
The next definition explains how to remedy this issue.

\begin{definition} 
Let $\iota_n$ denote the inclusion $A_n (X; R) \hookrightarrow C_n (X; R)$.
Let $\Omega_0 (X; R) = A_0 (X; R)$. 
For all $n > 0$, define $\Omega_n (X; R)$ to be the pullback of the diagram:
\begin{center}
\begin{tikzcd}[column sep = large]
\Omega_n (X; R) \ar[r] \ar[d] \ar[rd, phantom, "\lrcorner" very near start] & A_{n-1}(X; R)\dar["\iota_{n-1}", hookrightarrow]\\
A_n (X; R)\rar["\bd_n\circ \iota_n"]& C_{n-1}(X; R)
\end{tikzcd}
\end{center}
\end{definition}

Explicitly, the elements in $\Omega_n (X; R)$ consist of pairs  $(a,b) \in  A_n (X ; R) \times A_{n-1} (X ; R)$ such that $(\bd_n \iota_n)(a)=\iota_{n-1}(b)$ in $C_{n-1} (X ; R)$.

\begin{definition}
    For each $n > 0$, let $\bd_n\colon \Omega_n (X; R) \to\Omega_{n-1} (X; R)$ be the map $\bd_n(a,b) := (b,0)$.
\end{definition}

\begin{lemma}[{\cite[\S2.4]{grigor'yan-lin-muranov-yau:unpublished}}]
    With the definitions above, $(\Omega_\bullet (X; R), \bd_\bullet)$ is a chain complex.
\end{lemma}
\begin{proof}
    We have $\bd_{n-1} (\bd_n (a,b)) = \bd_{n-1} (b,0) = (0,0)$ i.e. $\bd^2 = 0$.
\end{proof}

\begin{definition}
The \emph{path homology} of a digraph $X$ with $R$-coefficients, denoted $H_\bullet (X ; R)$, is the homology of the chain complex $(\Omega_\bullet (X ; R), \bd_\bullet)$.
\end{definition}

\begin{definition} \label{def:homology-iso}
    Given a ring $R$, a digraph map $f \from X \to Y$ is an \emph{$R$-homology isomorphism} if for all $n \geq 0$, the induced map $f_* \from H_n(X; R) \to H_n(Y; R)$ is an isomorphism of $R$-modules.
\end{definition}

If the ring $R$ is clear from context or the statement is true for an arbitrary coefficient ring, we will speak of just \emph{homology isomorphisms}.
Likewise, we often write $\Omega_n (X)$, $H_n (X)$, etc., when the coefficient ring $R$ is clear from context.
Our main goal (\cref{cofib-cat}) is to show that homology isomorphisms (for any ring $R$) are a `convenient' class of weak equivalences in that they are a part of a cofibration category structure on $\DiGraph$.

For the benefit of the readers unfamiliar with path homology, we now compute some examples of path homology by hand.
Coefficients are in a general ring $R$.

\begin{example}\label{ex:homology_In}
Let $I_2$ be as in \cref{def:I_n}. The regular allowed paths in $I_2$ as well as representatives for elements in $\Omega_\bullet (I_2)$ can be found in the table below.
Although $012$ is present in $A_2$, its boundary is $12 - 02 + 01$ which does not land in $A_1$ since $02$ is not an allowed path in $A_1$.
Thus $\Omega_2 = 0$.

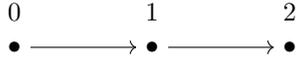
\begin{figure}[H]
\centering
 \begin{subfigure}[c]{0.4\textwidth}
        \begin{tikzpicture}[node distance=40pt]        
            \node(0) {$\bullet$};
            \node(1) [right=of 0]{$\bullet$};
            \node(2) [right=of 1]{$\bullet$};
            
            \node[above=1pt of 0] {$0$};
            \node[above=1pt of 1] {$1$};
            \node[above=1pt of 2] {$2$};
            
            \draw[->] (0) to (1);
            \draw[->] (1) to (2);
        \end{tikzpicture}
\end{subfigure}
\begin{subfigure}[c]{0.5\textwidth}
    \centering
\begin{tabular}{||c||  c | c | c ||} 
 \hline
 $l$ &  $A_l$ & $\Omega_l$ & $H_l$ \\ [.5ex] 
 \hline\hline
0 &  0 \, 1 \, 2  & 0 \, 1 \, 2 & $R$  \\ 
 \hline
 1 &  01 \, 12  & 01 \, 12 & 0  \\
 \hline
 $2$ & 012 & $\varnothing$ & 0 \\
 \hline
 $\geq 3$ & $\varnothing$ & $\varnothing$ & 0 \\
 \hline
\end{tabular}
\end{subfigure}
\end{figure}
\end{example}

The previous example is a simple case of the following general result:

\begin{lemma}[{\cite[Cor.~4.6]{grigor'yan-lin-muranov-yau}}]\label{rmk:tree_homology} 
If the underlying graph of $X \in \DiGraph$ is a tree, then $\Omega_l(X) = 0$ for all $l \geq 2$, and $X$ has trivial path homology above degree $1$.
\end{lemma}

\begin{example}[Oriented triangle]\label{ex:homology_C3} 
Consider the cycle graph $C_3$ as in \cref{def:Cn}.

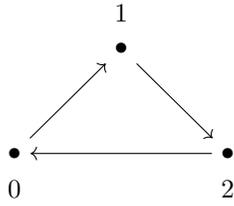
\begin{figure}[H]
\centering
 \begin{subfigure}[c]{0.3\textwidth}
        \begin{tikzpicture}[node distance=40pt]      
            \node(1) {$\bullet$};
            \node(0)[below left = of 1] {$\bullet$};
            \node(2)[below right=of 1]{$\bullet$};
            \node[above =1pt of 1]{$1$};
            \node[below =1pt of 2]{$2$};
            \node[below =1pt of 0]{$0$};
            \draw[->] (2) to (0);
            \draw[->] (0) to (1);
            \draw[->] (1) to (2);
        \end{tikzpicture}
\end{subfigure}
\begin{subfigure}[c]{0.6\textwidth}
    \centering
         \begin{tabular}{||c||  c | c | c ||} 
         \hline
        $l$ &  $A_l$ & $\Omega_l$ & $H_l$ \\ [.5ex] 
         \hline\hline
        0 &  0 \, 1 \, 2 & 0 \, 1 \, 2 & $R$  \\ 
        \hline
        1 &  01 \, 12 \, 20  & 01 \, 12 \, 20 & $R$ \\
        \hline
        2 & 012 \, 120 \, 201 & $\varnothing$ & 0 \\
        \hline
        3 & 0120 \, 1201 \, 2012 & $\varnothing$ & 0 \\
        \hline
        $l \geq 4$  & $01\cdots l$\,\,\,\,\,  $12\cdots l\!+\!1$\,\,\,\,\,  $20\cdots l\!+\!2$& $\varnothing$ & 0 \\ [1ex] 
        \hline
        \end{tabular}
    \end{subfigure}
\end{figure}

The image of $\bd_1: \Omega_1(C_3) \to \Omega_0(C_3)$ is generated by $0-1,\, 1-2,$ and the kernel of $\bd_1$ is generated by $01+12+20.$ 
To see that $\Omega_l$ is zero for $l\geq 2$, 
note that the boundary $\bd_l$ of $(i, i+1,i+2, \ldots ,i+l)$ includes a nonzero summand $$(i,i+2,\ldots, i+l)\in R\{(C_3)_V^l\},$$ which is not an allowed path and is not a summand of the boundary of any of the other generator $(j,j+1,j+2,\ldots, j+l)$ for $A_l(C_3).$ 
\end{example}

\begin{example}[Commuting triangle] \label{ex:commuting_triangle_homology}
Consider $C_{2,1}$ as in \cref{def:Cmn}.

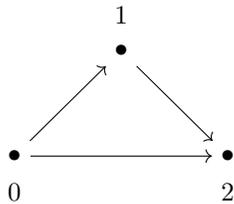
 \begin{figure}[H]
        \centering
 \begin{subfigure}[c]{0.3\textwidth}
        \begin{tikzpicture}[node distance=40pt]      
            \node(1) {$\bullet$};
            \node(0)[below left = of 1] {$\bullet$};
            \node(2)[below right=of 1]{$\bullet$};
            \node[above =1pt of 1]{$1$};
            \node[below =1pt of 2]{$2$};
            \node[below =1pt of 0]{$0$};
            \draw[->] (0) to (2);
            \draw[->] (0) to (1);
            \draw[->] (1) to (2);
        \end{tikzpicture}
\end{subfigure}
\begin{subfigure}[c]{0.6\textwidth}
        \centering
    \begin{tabular}{||c||  c | c| c ||} 
 \hline
 $l$ &  $A_l$ & $\Omega_l$ & $H_l$ \\ [.5ex] 
 \hline\hline
0 &  0 \, 1 \, 2 & 0 \, 1 \, 2  & $R$\\ 
 \hline
 1 &  01 \, 12 \, 02 & 01 \, 12 \, 02 & 0  \\
 \hline
 2 & 012 & 012 & 0\\
 \hline
 $\geq 3$ & $\varnothing$ & $\varnothing$ & 0 \\
 \hline
\end{tabular} 
\end{subfigure}
    \end{figure}  
The image of $\bd_1:\Omega_1(C_{2,1}) \to \Omega_0(C_{2,1})$ is generated by $1-0, 2-1$ so the kernel of $\bd_2$ on $\Omega_\bullet (C_{2,1})$ is one-dimensional. Its generator is given by $\bd_2(201)=12-02+01$.
\end{example}

\begin{example}[Commuting Square]\label{ex:commuting_square} 
Consider $C_{2,2}$ as in \cref{def:Cmn}.

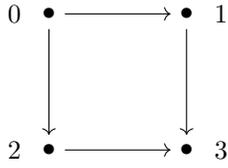
    \begin{figure}[H]
        \centering
        \begin{subfigure}[c]{0.3\textwidth}
        \begin{tikzpicture}[node distance=40pt]        
            \node(0) {$\bullet$};
            \node(1) [right= of 0] {$\bullet$};
            \node(3) [below=of 1] {$\bullet$};
            \node(2) [below=of 0] {$\bullet$};

            \node[left=1pt of 0] {$0$};
            \node[right=1pt of 1] {$1$};
            \node[right=1pt of 3] {$3$};  
            \node[left=1pt of 2] {$2$}; 

            \draw[->] (0) to (1);
            \draw[->] (1) to (3);
            \draw[->] (0) to (2);
            \draw[->] (2) to (3);

        \end{tikzpicture}
        \end{subfigure}
        \begin{subfigure}[c]{0.6\textwidth}
        \centering
        \begin{tabular}{||c||  c | c | c ||} 
         \hline
         $l$ &  $A_l$ & $\Omega_l$ & $H_l$ \\ [.5ex] 
         \hline\hline
        0 &  0 \, 1 \, 2 \,3 & 0 \, 1 \, 2 \,3 & $R$ \\ 
         \hline
          1 & 01\,\,\,\, 13\,\,\,\, 02 \,\,\,\,23 & 01\,\,\,\, 13\,\,\,\, 02 \,\,\,\,23 & 0\\ 
         \hline
          2 & 013\,\,\,\, 023 & 013 - 023 & 0\\ 
         \hline
          $\geq 3$ & $\varnothing$ & $\varnothing$ & 0 \\
         \hline
        \end{tabular}
        \end{subfigure}
    \end{figure}  
Despite looking like a the topological circle $S^1$, $H_1$ of this graph is $0$.
Individually, the boundaries of the allowed paths $013$ and $023$ are not in $A_1$, since $A_1$ does not contain $03$.
However, the boundary of the linear combination $013 - 023$ \emph{does} land in $A_1$.
We thus have a single non-zero element in $\Omega_2$ whose boundary generates the kernel of $\bd_1$ and $H_1 = 0$. 
\end{example}

\begin{example}\label{ex:3_1_square}
Consider the cycle $C_{3,1}$ as in \cref{def:Cmn}.

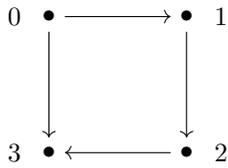
    \begin{figure}[H]
        \centering
        \begin{subfigure}[c]{0.3\textwidth}
            \begin{tikzpicture}[node distance=40pt]
                \node(0) {$\bullet$};
                \node(1) [right=of 0] {$\bullet$};
                \node(2) [below=of 1] {$\bullet$};
                \node(3) [below=of 0] {$\bullet$};
                
                \node[left=1pt of 0] {$0$};
                \node[right=1pt of 1] {$1$};
                \node[right=1pt of 2] {$2$};   
                \node[left=1pt of 3] {$3$};   
                
                \draw[->] (0) to (1);
                \draw[->] (1) to (2);
                \draw[->] (2) to (3);
                \draw[->] (0) to (3);
            \end{tikzpicture}
        \end{subfigure}
        \begin{subfigure}[c]{0.6\textwidth}
        \centering
        \begin{tabular}{||c||  c | c | c ||} 
         \hline
         $l$ &  $A_l$ & $\Omega_l$ & $H_l$ \\ [.5ex] 
         \hline\hline
        0 &  0 \, 1 \, 2 \,3 & 0 \, 1 \, 2 \, 3 & $R$ \\ 
         \hline
         1 &  01 \, 12 \, 23 \, 03 & 01 \, 12 \, 20 \,  03 & $R$ \\
         \hline
          2 & 012 \, 123 & $\varnothing$ & $0$ \\
         \hline
          3 & 0123 & $\varnothing$ & 0\\
         \hline
        \end{tabular}
        \end{subfigure}
    \end{figure}  Unlike the previous example, for $l\geq 2$, there are no linear combinations of elements in $A_l$ whose boundaries lie in $A_{l-1}$.
Hence $\Omega_l = 0$ for $l \geq 2$.
The kernel of $\bd_1$ is 1-dimensional and the image of $\bd_2$ is $0$, so $H_1 = R$.
\end{example}

\begin{remark}\label{rmk:homology_cycles} \cref{ex:homology_C3,ex:commuting_triangle_homology,ex:commuting_square,ex:3_1_square} are explicit cases of \cite[Prop.~4.3]{grigor'yan-lin-muranov-yau}.
A cycle graph with at least three vertices and any orientation on edges has the homology type of $S^1$ unless it is the commuting triangle (\cref{ex:commuting_triangle_homology}) or the commuting square (\cref{ex:commuting_square}).
\end{remark}
The previous examples may give the impression that the homology of a digraph is always trivial above degree 1, but this is not the case. 
\begin{example}\label{ex:suspension} 
Let $S\tilde C_4$ denote the digraph with vertices $a, 0, 1, 2, 3, b$ as depicted in the diagram below, where $1$ and $3$ are the two unlabelled vertices (it does not matter which):

    
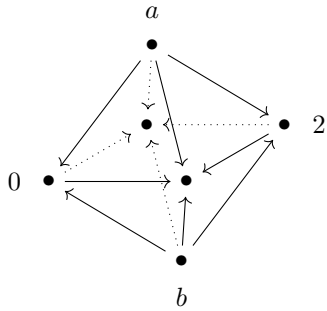
\begin{figure}[H]
        \centering
        \begin{subfigure}[c]{0.3\textwidth}
        \begin{tikzpicture}[node distance=40pt]        
            \node(0) {$\bullet$};
            \node(1) [above right=10pt and 25pt of 0] {$\bullet$};
            \node(2) [right=of 1] {$\bullet$};
            \node(3) [right=of 0] {$\bullet$};
            \node(a) [above left=40pt and 1pt of 3] {$\bullet$};
            \node(b) [below right=40pt and 1pt of 1] {$\bullet$};

            \node[left=1pt of 0] {$0$};
            \node[right=1pt of 2] {$2$};  
            \node[above=1pt of a] {$a$};
            \node[below=1pt of b] {$b$};

            \draw[->, dotted] (0) to (1);
            \draw[->, dotted] (2) to (1);
            \draw[->] (2) to (3);
            \draw[->] (0) to (3);
            
            \draw[->] (a) to (0);
            \draw[->, dotted] (a) to (1);
            \draw[->] (a) to (2);
            \draw[->] (a) to (3);

            \draw[->] (b) to (0);
            \draw[->, dotted] (b) to (1);
            \draw[->] (b) to (2);
            \draw[->] (b) to (3);
        \end{tikzpicture}
        \end{subfigure}
        \begin{subfigure}[c]{0.6\textwidth}
        \centering
        \begin{tabular}{||c||  c | c | c ||} 
         \hline
         $l$ &  $A_l$ & $\Omega_l$ & $H_l$ \\ [.5ex] 
         \hline\hline
        0 &  a \, 0 \, 1 \, 2 \,3\, b & a \, 0 \, 1 \, 2 \,3\, b & $R$ \\ 
         \hline
          & a0\,\,\,\, a1\,\,\,\, a2 \,\,\,\,a3 & a0\,\,\,\, a1\,\,\,\, a2 \,\,\,\,a3 & \\ 
          1 & 01\,\,\,\, 21\,\,\,\, 23 \,\,\,\,03 & 01\,\,\,\, 21\,\,\,\, 23 \,\,\,\,03 & 0\\ 
          & b0 \,\,\,\, b1\,\,\,\, b2\,\,\,\, b3 & b0 \,\,\,\, b1\,\,\,\, b2\,\,\,\, b3 &  \\
         \hline
          \multirow{2}{*}{2} & a01\,\,\,\, a21\,\,\,\, a23 \,\,\,\,a03 & a01\,\,\,\, a21\,\,\,\, a23 \,\,\,\,a03 & \multirow{2}{*}{$R$}\\
            & b01\,\,\,\, b21\,\,\,\, b23 \,\,\,\,b03 & b01\,\,\,\, b21\,\,\,\, b23 \,\,\,\,b03 & \\
         \hline
          $\geq 3$ & $\varnothing$ & $\varnothing$ & 0 \\
         \hline
        \end{tabular}
        \end{subfigure}
    \end{figure}  

Note that all length two allowed paths contribute to $\Omega_2$ in this case. The image of $\bd_1$ is $5$-dimensional so the kernel of $\bd_1$ is $7$-dimensional.
The image of $\bd_2$ is also $7$-dimensional, so $H_1 = 0$.
The kernel of $\bd_2$ is generated by $a01-a21+a23-a03-b01+b21-b23+b03$ and, since $\Omega_3=0$, $H_2 = R$. 
\end{example}

\begin{remark} 
The graph of \cref{ex:suspension} above can be thought of a suspension of $\tilde C_4$ and the result of the computation is closely related to \cite[Prop.~5.10]{grigor'yan-lin-muranov-yau}.
\end{remark}

We have written a Python script to compute the dimensions of path homology over $\mathbb{R}$, which may be found at \cite{python_script}.
The script first generates the matrix representing the map:
    \[
        \begin{tikzcd}
            A_n X \rar["\bd_n \circ \iota_n"] &
            C_{n-1} X \rar["\pi_{n-1}"] &
            C_{n-1} X / A_{n-1} X
        \end{tikzcd}
    \]  
It then computes a basis for the nullspace of this matrix, which is a basis of $\Omega_n$ by the following lemma:

\begin{lemma} \label{omega_kernel}
    $\Omega_n X$ is the kernel of the map $A_n X \to C_{n-1} X/ A_{n-1} X$, i.e. we have an exact sequence:
    \[
        \begin{tikzcd}
            0 \rar & 
            \Omega_n X \rar[hookrightarrow] &
            A_n X \rar &
            C_{n-1} X / A_{n-1} X
        \end{tikzcd}
    \]  
\end{lemma}

\begin{proof}
    We first note that we have an exact sequence
    \[
        \begin{tikzcd}
            0 \rar & 
            A_{n-1} X \rar["\iota_{n-1}",hookrightarrow] &
            C_{n-1} X \rar["\pi_{n-1}"] &
            C_{n-1} X / A_{n-1} X  \rar &
            0
        \end{tikzcd}
    \]     
    This implies that the square on the right in the following diagram is a pullback (and a pushout, but we will not need this fact):
    \[
        \begin{tikzcd}
            \Omega_n \rar \dar[hookrightarrow]  \ar[dr, phantom, very near start, "\lrcorner"]
            &
            A_{n-1} \dar["\iota_{n-1}"', hookrightarrow] \rar \ar[dr, phantom, very near start, "\lrcorner"]
            & 
            0 \dar
            \\
            A_n \rar["\bd_n \circ \iota_n"']
            &
            C_{n-1} \rar["\pi_{n-1}"'] 
            & 
            C_{n-1}/A_{n-1}
        \end{tikzcd}
    \]
    We thus have a composite of two pullback squares, which is itself a pullback.
    The outer pullback square gives our desired exact sequence.
\end{proof}

The matrix representing the differential $\bd_n \colon \Omega_n X \to \Omega_{n-1} X$ is given by restricting the map $A_n X \xrightarrow{\iota_n} C_n X \xrightarrow{\bd_n} C_{n-1} X$ to $\Omega_n X \to \Omega_{n-1} X$ and expressing its matrix in terms of the bases for $\Omega_n X$ and $\Omega_{n-1} X$.
We compute the rank and nullity for the matrices of $\bd_1, \bd_2, \dots, \bd_K$, where $K$ is some pre-defined cut-off.
The nullity of $\bd_0$ is defined to be $\dim \Omega_{0} X$.
The dimensions of $H_n(X)$ for $0 \leq n < k$ are then given by
\[
    \dim H_n(X) = \text{nullity } \bd_n - \text{rank } \bd_{n+1}.
\]

\begin{example}
    Consider the subgraph obtained by removing the central vertex of $I_2 \gtimes I_2 \gtimes I_2$:
    
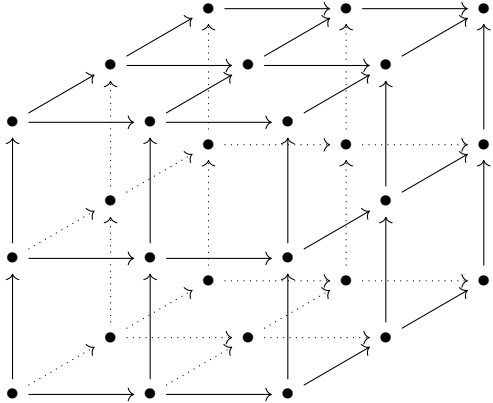
\begin{figure}[H]
        \centering
        \begin{subfigure}[c]{0.5\textwidth}
        \begin{tikzpicture}[node distance=40pt]        
            \node(a) {$\bullet$};
            \node(b) [right=of a] {$\bullet$};
            \node(c) [right=of b] {$\bullet$};
            \node(d) [above=of a] {$\bullet$};
            \node(e) [right=of d] {$\bullet$};
            \node(f) [right=of e] {$\bullet$};
            \node(g) [above=of d] {$\bullet$};
            \node(h) [right=of g] {$\bullet$};
            \node(i) [right=of h] {$\bullet$};
            
            \node(j) [above right=10pt and 25pt of a]{$\bullet$};
            \node(k) [right=of j] {$\bullet$};
            \node(l) [right=of k] {$\bullet$};
            \node(m) [above=of j] {$\bullet$};
            \node(n) [above=of l] {$\bullet$};
            \node(o) [above=of m] {$\bullet$};
            \node(p) [right=of o] {$\bullet$};
            \node(q) [right=of p] {$\bullet$};     

            \node(r) [above right=10pt and 25pt of j]{$\bullet$};
            \node(s) [right=of r]{$\bullet$};
            \node(t) [right=of s] {$\bullet$};
            \node(u) [above=of r] {$\bullet$};
            \node(v) [right=of u] {$\bullet$};
            \node(w) [right=of v] {$\bullet$};
            \node(x) [above=of u] {$\bullet$};
            \node(y) [right=of x] {$\bullet$};
            \node(z) [right=of y] {$\bullet$};
            
            \draw[->] (a) to (b);
            \draw[->] (b) to (c);
            \draw[->] (d) to (e);
            \draw[->] (e) to (f);
            \draw[->] (g) to (h);
            \draw[->] (h) to (i);
            
            \draw[->] (a) to (d);
            \draw[->] (d) to (g);
            \draw[->] (b) to (e);
            \draw[->] (e) to (h);
            \draw[->] (c) to (f);
            \draw[->] (f) to (i);
            
            \draw[->, dotted] (j) to (k);
            \draw[->, dotted] (k) to (l);
            \draw[->] (o) to (p);
            \draw[->] (p) to (q); 
            
            \draw[->, dotted] (j) to (m);
            \draw[->, dotted] (m) to (o);
            \draw[->] (l) to (n);
            \draw[->] (n) to (q);
            
            \draw[->, dotted] (r) to (s);
            \draw[->, dotted] (s) to (t);
            \draw[->, dotted] (u) to (v);
            \draw[->, dotted] (v) to (w);
            \draw[->] (x) to (y);
            \draw[->] (y) to (z);
            
            \draw[->, dotted] (r) to (u);
            \draw[->, dotted] (u) to (x);
            \draw[->, dotted] (s) to (v);
            \draw[->, dotted] (v) to (y);
            \draw[->] (t) to (w);
            \draw[->] (w) to (z);
            
            \draw[->, dotted] (j) to (r);
            \draw[->, dotted] (a) to (j);
            \draw[->, dotted] (m) to (u);
            \draw[->, dotted] (d) to (m);
            \draw[->] (o) to (x);
            \draw[->] (g) to (o);
            
            \draw[->] (l) to (t);
            \draw[->] (c) to (l);
            \draw[->] (n) to (w);
            \draw[->] (f) to (n);
            \draw[->] (q) to (z);
            \draw[->] (i) to (q);
            
            \draw[->, dotted] (k) to (s);
            \draw[->, dotted] (b) to (k);
            \draw[->] (p) to (y);
            \draw[->] (h) to (p);
        \end{tikzpicture}
        \end{subfigure}
        \begin{subfigure}[c]{0.4\textwidth}
        \centering
        \begin{tabular}{||c||  c | c||} 
         \hline
         $l$ &  dim $\Omega_l$ & $H_l$ \\ [.5ex] 
         \hline\hline
        0 &  26 & $R$ \\ 
         \hline
         1 & 48 & 0 \\
         \hline
          2 & 24 & $R$ \\
         \hline
          3 & 0 & 0 \\
           \hline
          4 & 0 & 0 \\
         \hline
        \end{tabular}
        \end{subfigure}
    \end{figure}  
    Using our script, we calculate the dimensions of $\Omega_\bullet$ and $H_\bullet$, which are displayed in the table above.
    The 24 small square faces assemble to form a 2-dimensional ``hole'' that remains unfilled, since $\Omega_3$ is 0.
    This is witnessed by the non-zero $H_2$.
    By contrast $H_2(I_2 \gtimes I_2 \gtimes I_2) = 0$.
\end{example}

\subsection*{Cofibration categories}

We now introduce cofibration categories, a categorical framework for studying abstract homotopy theory.
The origin of this notion, or more precisely its formal dual, goes back to Brown's \emph{categories of fibrant objects} \cite{brown:abstract-homotopy-theory}, a notion that was introduced to study generalized sheaf cohomology.
Many variations on the definition have appeared since, notably in Baues' book \cite{baues:algebraic-homotopy} and Radulescu-Banu's Ph.D.~thesis \cite{radulescu-banu}.

Cofibration categories provide a way of speaking about homotopy theories with (finite) homotopy colimits.
This statement was made precise by Szumi{\l}o \cite{szumilo:two-models}, who showed that the homotopy theory of cofibration categories is equivalent to the homotopy theory of (finitely) cocomplete $(\infty, 1)$-categories.

\begin{definition}\label{cofib-cat-def}
  A \emph{cofibration category} consists of a category $\C$ together with two classes of maps in $\C$: \emph{cofibrations}, denoted $\cto$, and \emph{weak equivalences}, denoted $\weto$, subject to the following conditions (where by an \emph{acyclic cofibration} we mean a morphism that is both a cofibration and a weak equivalence):
  \begin{enumerate}
      \item[(C1)] For any object $X \in \C$, the identity map $\id[X]$ is an acyclic cofibration.
      Both cofibrations and weak equivalences are closed under composition.
      \item[(C2)] The class of weak equivalences satisfies the 2-out-of-6 property, i.e., given a triple of composable morphisms $f \from X \to Y$, $g \from Y \to Z$, and $h \from Z \to W$, if $gf$ and $hg$ are weak equivalences, then so are $f$, $g$, $h$, and $hgf$.
      \item[(C3)] The category $\C$ admits an initial object $\varnothing$ and for any object $X \in C$, the unique map $\varnothing \to X$ is a cofibration (i.e., all objects are \emph{cofibrant}).
      \item[(C4)] The category $\C$ admits pushouts along cofibrations.
      Moreover, the pushout of an (acyclic) cofibration is an (acyclic) cofibration.
      \item[(C5)] For any object $X \in \C$, the codiagonal map $X \sqcup X \to X$ can be factored as a cofibration followed by a weak equivalence.
      \item[(C6)] The category $\C$ has small coproducts.
      \item[(C7)] The transfinite composite of (acyclic) cofibrations is again an (acyclic) cofibration.
  \end{enumerate}
\end{definition}

The above definition most closely resembles the one given in \cite{szumilo:two-models} and is a slight strengthening of what might be found in \cite{brown:abstract-homotopy-theory,baues:algebraic-homotopy}.
For instance, we require in (C2) that weak equivalences satisfy the 2-out-of-6 property instead of the (perhaps more common) 2-out-of-3 property.
We recall that 2-out-of-6 implies 2-out-of-3.

\begin{lemma}
  Weak equivalences in any cofibration category satisfy the 2-out-of-3 property, i.e., given a composable pair of maps $f \from X \to Y$ and $g \from Y \to Z$, if any two of $f$, $g$, $gf$ are weak equivalences, then so is the third. 
  \qed
\end{lemma}


Furthermore, in the presence of axioms (C1), (C3), (C4), and (C5), 2-out-of-6 is equivalent to 2-out-of-3 with an additional requirement that the class of weak equivalences is saturated, i.e., that maps inverted when passing to the homotopy category are exactly the weak equivalences (this result is due to Cisinski, cf.~\cite[Thm.~7.2.7]{radulescu-banu}).

The seven axioms presented above fall into two categories: axioms (C1)--(C5) correspond to finite homotopy colimits, whereas axioms (C6) and (C7) say that the homotopy theory additionally admits infinite homotopy colimits.

Before discussing examples of cofibration categories, we briefly record two consequences of the axioms.
The first shows that the axiom (C5) can be strengthened to ask for factorizations of arbitrary maps rather than just the codiagonal morphism.

\begin{lemma}[Factorization Lemma, {\cite[p.~421]{brown:abstract-homotopy-theory}}]
  Every map $f$ in $\C$ can be factored as $f = w i$ where $i$ is a cofibration and $w$ is a weak equivalence.
\end{lemma}

\begin{lemma}[Left Properness, {\cite[Lem.~I.4.2]{brown:abstract-homotopy-theory}}] \label{left-properness}
  The pushout of a weak equivalence along a cofibration is again a weak equivalence.
\end{lemma}

Readers familiar with homotopical algebra may recoginize that a large class of examples of cofibration categories come from model categories.
The latter, introduced by Quillen \cite{quillen:book}, were the first known way of abstractly capturing what a ``homotopy theory'' is.
In brief, a model category is a complete and cocomplete category $\C$ equipped with three classes of morphisms: cofibrations, fibrations, and weak equivalences subject to axioms similar to, although stronger than, the ones of a cofibration category.
Given a model category $\C$, its full subcategory of cofibrant objects (i.e., objects $X$ for which the map $\varnothing \to X$ is a cofibration) is a cofibration category.
However, we shall not be concerned with model categories in this paper.

\begin{example} \label{ex:top-cofib-cat}
  The category $\Top$ of topological spaces and some of its subcategories carry several interesting cofibration category structures:
  \begin{itemize}
      \item The \emph{Hurewicz cofibration category} structure is defined on the category of all spaces.
      Its weak equivalences are the homotopy equivalences and its cofibrations are Hurewicz cofibrations, i.e., maps $i \from A \cto X$ such that, for any space $S$, any commutative square of the form
      \[ \begin{tikzcd}
        A \arrow[r] \arrow[d,tail] & S^{[0,1]} \arrow[d] \\
        X \arrow[r] & S \\
      \end{tikzcd} \]
      where the right hand map is the evaluation at $0$ admits a diagonal filler, i.e., there is a map $X \to S^{[0,1]}$ making both triangles commute.
      (This cofibration category structure arises from the Hurewicz model structure on $\Top$, cf.~\cite[Thm.~3]{strom}.)
      \item The \emph{Serre cofibration category} structure is defined on the category of retracts of CW-complexes.
      Its weak equivalences are weak homotopy equivalences, i.e., maps inducing isomorphisms on homotopy groups, and its cofibrations are retracts of CW-inclusions.
      (This cofibration category structure arises from the Serre model structure on $\Top$, cf.~\cite[Thm.~2.4.19]{hovey:book}.)
      \item The \emph{Dold cofibration category} structure is defined on the category of all spaces.
      Its weak equivalences are the homotopy equivalences and its cofibrations are \emph{Dold cofibrations}, i.e., maps $A \cto X$ satisfying the following weak homotopy extension condition: for any space $S$, every commutative square of the form
      \[ \begin{tikzcd}
        A \arrow[r] \arrow[d,tail] & S^{[0,1]} \arrow[d] \\
        X \arrow[r] & S \\
      \end{tikzcd} \]
      admits a diagonal filler making the upper triangle commute strictly and the lower triangle commute up to a homotopy relative to $A$.
      (This cofibration category structure does {\em not} arise from a model structure, cf.~\cite[Thm.~1.32.(2)]{szumilo:two-models}.)
  \end{itemize}
\end{example}

\begin{example}\label{Ch-cof-cat}
  For chain complexes $\Ch_R$ over a unital ring $R$:
  \begin{itemize}
      \item The \emph{injective cofibration category} structure is defined on all chain complexes $\Ch_R$. Its weak equivalences are quasi-isomorphisms, i.e., maps inducing isomorphisms on all homology groups, and its cofibrations are monomorphisms (cf.~\cite[Thm.~2.3.13]{hovey:book}).
      \item The \emph{projective cofibration category} structure is defined on chain complexes of projective $R$-modules $\Ch_R^\proj$. Its weak equivalences are once again the quasi-isomorphisms and its cofibrations are monomorphisms with degree-wise projective cokernels (cf.~\cite[Thm.~2.3.11]{hovey:book}).
  \end{itemize}
\end{example}

\begin{definition}
  A functor $F \from \C \to \D$ between cofibration categories is \emph{exact} if it preserves cofibrations, acyclic cofibrations, the initial object, pushouts along cofibrations, coproducts, and transfinite composites of cofibrations.
\end{definition}

\begin{remark} \label{rem:ken-brown}
    It follows by Ken Brown's Lemma \cite[Lem.~1.1.12]{hovey:book} that exact functors preserve weak equivalences.
\end{remark}

\begin{example}\label{Ch-inclusion-exact}
  The inclusion $\Ch_R^\proj \into \Ch_R$ is an exact functor from the projective cofibration category of chain complexes to the injective one.
\end{example}

\begin{example}
  The singular chain complex functor $C_\bullet  \from \Top \to \Ch_\mathbb{Z}$ taking a topological space to its singular complex is an exact functor.
  Here, $\Top$ is considered with the Serre cofibration category structure and $\Ch_\mathbb{Z}$ is considered with the injective cofibration category structure.
  It can also be noted that the functor takes values in the category $\Ch_\mathbb{Z}^\mathsf{proj}$ and could also be considered as an exact functor into the projective cofibration category structure.
\end{example}

\section{Cofibrations of directed graphs} \label{sec:cofibs}

In this section, we define a suitable notion of cofibration of directed graphs (\cref{cofib-def}) and prove some closure properties of these cofibrations (\cref{cof-wide-subcat} through \cref{closed-under-transf-comp}). 
Many of these properties correspond to cofibration category axioms.
Stability of acyclic cofibrations under pushout requires a proof of the excision axiom, to which we dedicate \cref{sec:excision}.
We defer the full proof that our cofibrations and weak equivalences comprise a cofibration category structure on $\DiGraph$ until \cref{sec:main-thm}. 
Unless otherwise noted, $X$ and $A$ (resp.~$Y$ and $B$, $X'$ and $A'$) will refer to a graph and an induced subgraph, respectively.

\subsection*{Projecting decompositions}

We begin this section by adapting the notion of projecting decomposition (cf.~\cite[Def.~4.6]{leinster:magnitude-graph} and \cite[Def.~26]{hepworth-willerton}) to the setting of directed graphs.


\begin{definition}
For a directed graph $X$ with an induced subgraph $A$, let $X^A$ denote the induced subgraph on the set of vertices of $X$ which admit a path to some vertex of $A$.     
\end{definition}

\begin{definition}
For a vertex $x$ of $X^A$, the \emph{height} of $x$, denoted $h(x)$, is the minimal length of a path from $x$ to a vertex of $A$.
\end{definition}

\begin{definition}\label{proj-decomp-def}
A \emph{projecting decomposition} of $X$ with respect to $A$ is a function $\pi \colon X_{V}^A \to A_V$ such that, for any $x \in X_V$ and any $a \in A_V$ admitting a path from $x$, there is a path from $x$ to $a$ of minimal length which passes through $\pi x$.
\end{definition}

\begin{example}\label{ex:proj-decomp}
Let $X$ be the cycle $C_{3,1}$ from \cref{def:Cmn}.
If $A$ is the edge $0-1$ (\Cref{C31-01}), the only vertices in $X$ admitting a path to $A$ are those in $A$ itself, so $X_{V}^A = A_V$ and the identity function is a projecting decomposition.

If $A$ is the edge $1-2$ (\Cref{C31-12}), $X^A$ is the induced subgraph on the vertices $0,1$ and $2$, with a projecting decomposition sending $0$ and $1$ to $1$, and $2$ to itself.

If $A$ is the edge $2-3$ (\Cref{C31-23}), $X^A$ is the whole of $X$. However, $X$ admits no projecting decomposition $\pi$ to $A$: if $\pi(0) = 2$, the unique path of minimal length $1$ from $0$ to $3$ does not pass through $2$; and if $\pi(0) = 3$, the unique path of minimal length $2$ from $0$ to $2$ does not pass through $3$.

    \begin{figure}[H]
        \centering
        \begin{subfigure}[c]{0.3\textwidth}
            \centering
            \begin{tikzpicture}[node distance=40pt, halo/.style={line join=round, double, line cap=round,double distance=18pt,shorten >=-6pt,shorten <=-6pt}]
                \node(0) [red] {$\bullet$};
                \node(1) [right=of 0, red] {$\bullet$};
                \node(2) [below=of 1] {$\bullet$};
                \node(3) [below=of 0] {$\bullet$};
                
                \node[left=1pt of 0, red] {$0$};
                \node[right=1pt of 1, red] {$1$};
                \node[right=1pt of 2] {$2$};   
                \node[left=1pt of 3] {$3$};   
                
                \draw[->, red] (0) to (1);
                \draw[->] (1) to (2);
                \draw[->] (2) to (3);
                \draw[->] (0) to (3);

                \begin{scope}[on background layer]
                \draw[halo, blue] (0) -- (1);   
                \end{scope}      
            \end{tikzpicture}
            \caption{}
            \label{C31-01}
        \end{subfigure}
        \begin{subfigure}[c]{0.3\textwidth}
            \centering
            \begin{tikzpicture}[node distance=40pt, halo/.style={line join=round, double, line cap=round,double distance=18pt,shorten >=-6pt,shorten <=-6pt}]
                \node(0) {$\bullet$};
                \node(1) [right=of 0, red] {$\bullet$};
                \node(2) [below=of 1, red] {$\bullet$};
                \node(3) [below=of 0] {$\bullet$};
                
                \node[left=2pt of 0] {$0$};
                \node[right=2pt of 1, red] {$1$};
                \node[right=2pt of 2, red] {$2$};   
                \node[left=2pt of 3] {$3$};   
                
                \draw[->] (0) to (1);
                \draw[->, red] (1) to (2);
                \draw[->] (2) to (3);
                \draw[->] (0) to (3);

                \begin{scope}[on background layer]
                \node(22) [below=44pt of 1] {};
                \node(11) [above=42pt of 2] {};                
                \draw[halo, blue] (22) -- (11) -- (1) -- (0);
                \end{scope}                
            \end{tikzpicture}
            \caption{}
            \label{C31-12}
        \end{subfigure}
        \begin{subfigure}[c]{0.3\textwidth}
            \centering
            \begin{tikzpicture}[node distance=40pt]
                \node(0) {$\bullet$};
                \node(00)[above left=1pt of 0] {};
                \node(1) [right=of 0] {$\bullet$};
                \node(2) [below=of 1, red] {$\bullet$};
                \node(22) [below right=1pt of 2] {};
                \node(3) [below=of 0, red] {$\bullet$};
                
                \node[left=3pt of 0] {$0$};
                \node[right=3pt of 1] {$1$};
                \node[right=3pt of 2, red] {$2$};   
                \node[left=3pt of 3, red] {$3$};   
                
                \draw[->] (0) to (1);
                \draw[->] (1) to (2);
                \draw[->, red] (2) to (3);
                \draw[->] (0) to (3);

                \begin{scope}[on background layer]
                    \draw[rounded corners=10pt, blue] (00) rectangle (22) {};
                \end{scope}                
            \end{tikzpicture}
            \caption{}
            \label{C31-23}
        \end{subfigure}
        \caption{The graph $X = C_{3,1}$ with induced subgraphs $A$ (in red) and corresponding $X^A$ (circled in blue). Only the first two examples admit projecting decompositions.}
    \end{figure} 

\end{example}

We now explore some basic consequences of the definition of a projecting decomposition.

\begin{lemma}\label{pi-closest}
If $X$ admits a projecting decomposition with respect to $A$, then for any $x \in X_{V}^A$, $\pi x$ is the unique vertex of $A$ which is closest to $x$, \ie{} the unique vertex of $A$ admitting a path of length $h(x)$.

In particular, if $x \in A_V$, then $\pi x = x$, while if $h(x) = 1$ then $\pi x$ is the unique vertex of $A$ admitting an edge from $x$.
\end{lemma}

\begin{proof}
Suppose that $a \in A_V$ admits a path of length $h(x)$ from $x$ to $a$. By the definition of a projecting decomposition, there exists a path of length $h(x)$ from $x$ to $a$ which passes through $\pi x$. As $\pi x \in A_V$, the minimality of $h(x)$ implies that $\pi x$ must be the final vertex of this path and in fact $\pi x = a$.
\end{proof}

\begin{corollary}
If $X$ admits a projecting decomposition with respect to $A$, then it is unique. \qed
\end{corollary}

\begin{remark}\label{rmk:projecting_alternative} We will sometimes view a projecting decomposition $\pi$ as a function on $X^A_V \setminus A_V$ rather than $X^A_V$, when this suits our computational purposes. This abuse of terminology is justified by \cref{pi-closest}, which implies that a function on $X^A_V \setminus A_V$ satisfying the criteria of \cref{proj-decomp-def} extends uniquely to a projecting decomposition of $X$ with respect to $A$, by setting $\pi a = a$ for all $a \in A_V$.
\end{remark}

\begin{lemma}\label{pi-edges}
Suppose $X$ admits a projecting decomposition with respect to $A$.
Let $x \to y$ denote an edge of $X$, with both $x$ and $y$ admitting paths to $A$, and $h(x) \geq h(y)$.
Then one of the following conditions holds:
\begin{enumerate}
    \item \label{pi-edges-equal} $h(x) = h(y)$ and there exists an edge $\pi x \to \pi y$;
    \item \label{pi-edges-greater} $h(x) = h(y) + 1$ and $\pi x = \pi y$.
\end{enumerate}
\end{lemma}

\begin{proof}
If $x = y$ then condition \ref{pi-edges-equal} is trivially satisfied, so assume otherwise. We first note that by \cref{pi-closest}, there exists a path of length $h(y)$ from $y$ to $\pi y$. Concatenating this path with the edge $x \to y$, we obtain a path from $x$ to $\pi y$ of length $h(y) + 1$. It thus follows that $h(x) \leq h(y) + 1$. 

Moreover, by the definition of a projecting decomposition, there is a path of minimal length from $x$ to $\pi y$ which passes through $\pi x$; the length of this path must be less than or equal to $h(y) + 1$. The assumption that $h(x) \geq h(y)$ then implies that its length is either $h(y)$ or $h(y) + 1$.

First, suppose the length of this path is $h(y)$. Then $h(x) \leq h(y)$; together with our assumption on $h(x)$ this implies $h(x) = h(y)$. By \cref{pi-closest} we see that $\pi x = \pi y$, so that condition \ref{pi-edges-equal} is satisfied in this case.

Next suppose the length of the chosen path from $x$ to $\pi y$ is $h(y) + 1$. If $\pi x = \pi y$, then this implies that $h(x) = h(y) + 1$, so that condition \ref{pi-edges-greater} is satisfied. Otherwise, consider the initial segment of this path from $x$ to $\pi x$. As the path's length is minimal, this initial segment must have length $h(x)$. By our assumption that $\pi x \neq \pi y$, this initial segment is not the entire path; thus we see that $h(x) \leq h(y)$, once again implying $h(x) = h(y)$. As the length of the entire path is $h(y) + 1$, the path consists of this initial segment concatenated with an edge from $\pi x$ to $\pi y$. Thus condition \ref{pi-edges-equal} is satisfied in this case.
\end{proof}

\subsection*{Definition of cofibrations}

We now define the cofibrations which will be our objects of study.

\begin{definition}\label{cofib-def}
A \emph{cofibration} of directed graphs is an induced subgraph inclusion $A \cto X$ satisfying the following two conditions:
\begin{itemize}
    \item there are no edges out of $A$, \ie{} no edges from the vertices of $A$ to those not contained in $A$;
    \item $X$ admits a projecting decomposition with respect to $A$.
\end{itemize}
\end{definition}

\begin{example}\label{ex:cofib}
The inclusion of the edge $2-3$ in the commuting square $C_{2,2}$ is a cofibration: there are no edges out of $2-3$, and the map $0 \mapsto 2, 1 \mapsto 3$ is a projecting decomposition.

    \begin{figure}[H]
        \centering
        \begin{tikzpicture}[node distance=40pt, halo/.style={line join=round, double, line cap=round,double distance=18pt,shorten >=-6pt,shorten <=-6pt}]
            \node(0) {$\bullet$};
            \node(1) [right=of 0] {$\bullet$};
            \node(2) [below=of 0, red] {$\bullet$};
            \node(3) [below=of 1, red] {$\bullet$};
            
            \node[left=1pt of 0] {$0$};
            \node[right=1pt of 1] {$1$};
            \node[left=1pt of 2, red] {$2$};   
            \node[right=1pt of 3, red] {$3$};   
            
            \draw[->] (0) to (1);
            \draw[->] (1) to (3);
            \draw[->] (0) to (2);
            \draw[->, red] (2) to (3);

        \end{tikzpicture}
        \label{C22-23}
        \caption{The inclusion of the edge $2-3$ (in red) into the commuting square $X = C_{2,2}$ is a cofibration.}
    \end{figure}

\end{example}

\begin{remark}\label{rmk:cofib-both-conditions}
One might ask whether a simpler cofibration category for path homology can be obtained using only one of the two conditions of \cref{cofib-def} to define cofibrations. In fact, both conditions are necessary.  

To see this, consider the inclusion of the edge $0-1$ or the edge $2-3$ into the $(3,1)$-cycle $C_{3,1}$. The edge $0-1$ (\cref{C31-01}) admits a projecting decomposition, as we saw in \cref{ex:proj-decomp}, but also admits edges out of itself. Meanwhile, the edge $2-3$ (\cref{C31-23}) does not admit edges out of itself, but also does not admit a projecting decomposition. These examples would be cofibrations if we omit either of the conditions in the definition. 

In each case, consider the pushout of the inclusion of the edge in question along the homology isomorphism $I_1 \to I_0$ (see \cref{rmk:tree_homology}).
The pushouts are both isomorphic in $\DiGraph$ to $C_{2,1}$ of \cref{def:Cmn}.
The map from $C_{3,1}$ to the pushout cannot be a homology isomorphism (\cref{ex:commuting_triangle_homology,ex:3_1_square}).
Thus the proposed cofibration category structure does not exist, as it fails left properness (\cref{left-properness}).
\end{remark}

The remainder of this section will be concerned with proving certain useful properties of cofibrations.
In particular, many of these results will be used in \cref{sec:main-thm} to establish a cofibration category structure on $\DiGraph$, with cofibrations as defined above and path homology isomorphisms as the weak equivalences.

\begin{proposition}\label{cof-wide-subcat}
The class of cofibrations contains all identities and is closed under composition.
\end{proposition}

\begin{proof}
To see that all identities are cofibrations, we note that in the case $A = X$, it is trivially true that there are no edges out of $A$ and a projecting decomposition given by the identity on $A_V$.

Now let $A \cto X$ and $X \cto Y$ be a composable pair of cofibrations; we must show that the composite inclusion $A \cto Y$ is a cofibration. 

First, let $y$ denote a vertex of $Y \setminus A$ and let $a$ denote a vertex of $A$; we will show that there is no edge from $a$ to $y$. If $y \in X_V$ then this follows from the fact that $A \cto X$ is a cofibration, since there are no edges out of $A$ in $X$
; similarly, if $y \in Y_V\setminus X_V$ then this follows from the fact that $X \cto Y$ is a cofibration.

Now we construct a projecting decomposition of $Y$ with respect to $A$. Note that we already have projecting decompositions of $Y$ with respect to $X$ and of $X$ with respect to $A$; denote these by $\pi_X$ and $\pi_A$, respectively.

For a vertex $y \in Y_V^A$, let $a$ denote a vertex of $A$ admitting a path from $y$.
Then since $a \in X_V$, there is a path from $y$ to $a$ of minimal length passing through $\pi_X y$.
Observe that the terminal segment of this path beginning at $\pi_X y$ defines a path in $X$ of minimal length from $\pi_X y$ to $a$; thus we may replace this segment with one of equal length passing through $\pi_A \pi_X y$.
We can therefore define $\pi y$ to be $\pi_A \pi_X y$.
Note that this definition does not depend on the choice of $a \in A_V$ admitting a path from $y$.
It follows that $\pi$ is a projecting decomposition of $Y$ with respect to $A$, as for any vertex $a$ of $A$ admitting a path from $y$, we have defined a path of minimal length from $y$ to $a$ which passes through $\pi y$.
\end{proof}

\begin{proposition}\label{empty-cof}
For every directed graph $X$, the unique map $\varnothing \cto X$ is a cofibration. 
\end{proposition}

\begin{proof}
It is trivially true that there are no edges out of the empty subgraph of $X$. To obtain a projecting decomposition, we note that both $\varnothing_V$ and $A_V^\varnothing$ are empty, so a suitable function $\pi$ is given by the identity on the empty set.
\end{proof}

\subsection*{Closure properties of cofibrations}

In this subsection, we prove that cofibrations are closed under several natural operations: box products, pushouts, and retracts.
Several negative results are provided as well, including the failure of monoidality and non-existence of a small generating set of cofibrations.

\begin{proposition}\label{box-prod-cofibs}
A box product of cofibrations is a cofibration.
\end{proposition}

\begin{proof}
Let $A \cto X$ and $B \cto Y$ be a pair of cofibrations. Consider the box product $A \gtimes B \cto X \gtimes Y$. By the definition of the box product, we may observe that the vertices of $X \gtimes Y$ consist of all pairs $(x,y)$ with $x \in X_V, y \in Y_V$, and that $A \gtimes B$ is the induced subgraph on the set of such pairs for which $x \in A_V, y \in B_V$. Given two vertices $(x,y), (x',y')$, there is an edge $(x,y) \to (x',y')$ if and only if there are edges $x \to x'$ in $H$ and $y \to y'$ in $Q$, with at least one of these edges being degenerate (\ie{} either $x = x'$ or $y = y'$). 

It follows that a path from $(x,y)$ to $(x',y')$ is equivalent to an interleaving of a path from $x$ to $x'$ in $X$ with a path from $y$ to $y'$ in $Y$. More precisely, such a path consists of a sequence of edges of $X \gtimes Y$, each necessarily of the form $(e_i,\id)$ for $e_i$ an edge of $X$ or $(\id,e'_i)$ for $e'_i$ an edge of $Y$, where the sequence of edges $(e_i)$ forms a path from $x$ to $y$ in $X$ and the sequence of edges $(e'_i)$ forms a path from $x'$ to $y'$ in $Y$. The length of such a path is the sum of the lengths of its component paths. This characterization of edges in $X \gtimes Y$ shows that there are no edges out of $A \gtimes B$.

Now we define a projecting decomposition of $X \gtimes Y$ with respect to $A \gtimes B$. For $(x,y)$ admitting a path to a vertex of $A \gtimes B$, we define $\pi(x,y) = (\pi x, \pi y)$. To see that this is a projecting decomposition, consider a path in $X \gtimes Y$ from $(x,y)$ to $(a,b)$ where $a \in A_V, b \in B_V$. The characterization of paths above shows that this path is obtained by interleaving a path from $x$ to $a$ in $X$ with a path from $y$ to $b$ in $Y$. From these we obtain paths of minimal length from $x$ to $a$ through $\pi x$, and from $y$ to $b$ through $\pi y$. Choosing a suitable interleaving of these paths, we obtain a path of minimal length from $(x,y)$ to $(a,b)$ passing through $(\pi x, \pi y)$. Concretely, we may construct the desired path as follows:
\begin{itemize}
    \item proceed from $(x,y)$ to $(\pi x, y)$, moving at each step along the chosen path from $x$ to $a$ in the first component while keeping the second component fixed;
    \item proceed from $(\pi x, y)$ to $(\pi x, \pi y)$, moving at each step along the chosen path from $y$ to $b$ in the second component while keeping the first component fixed;
    \item proceed similarly from $(\pi x, \pi y)$ to $(a, \pi y)$;
    \item proceed similarly from $(a,\pi y)$ to $(a,b)$
\end{itemize}
Thus we see that $\pi$ is indeed a projecting decomposition.
\end{proof}

\begin{proposition}\label{cof-pushout}
Cofibrations are stable under pushout.
\end{proposition}

\begin{proof}
Consider a pushout diagram of directed graphs as depicted below, with $A \cto X$ a cofibration.
\[
\begin{tikzcd}
A \arrow[r,"f|_A"] \arrow[d,tail] \pushout & A' \arrow[d,hook] \\
X \arrow[r,"f"] & X' \\
\end{tikzcd}
\]
Applying \cref{pushout-char,pushout-complement}, we see that there are no edges out of $A'$ in $X'$ (as there are no edges out of $A$ in $X$), and that the complement of $A'$ in $X'$ is isomorphic to the complement of $A$ in $X$.

We now define a projecting decomposition $\pi'$ of $X'$ with respect to $A'$. Note that, by \cref{pushout-complement}, a vertex of $X' \setminus A'$ admits a path to $A'$ if and only if the corresponding vertex of $X \setminus A$ admits a path to $A$. Thus we define $\pi'$ as follows: for $a \in A'_V$ we set $\pi' a = a$, while for $x \in (X')^{A'}_V \setminus A'_V$ we set $\pi' x = f(\pi x)$.

It remains to show that $\pi'$ is in fact a projecting decomposition. Take $x \in (X')_V^{A'}\setminus A'_V$ and 
let $p$ be a path in $X'$ from $x$ to $a \in A'$. Without loss of generality assume that $p$ is of minimal length among such paths. Since there are no edges out of $A'$, the path $p$ may be depicted as in the following diagram, where $y_i \in X'_V \setminus A'_V$ for all $i$ and $b_j \in A'_V$ for all $j$:

\begin{figure}[H]
    \centering
    \begin{tikzpicture}[node distance=20pt]
        \node(0) {$\bullet$};
        \node(1) [right=of 0] {$\bullet$};
        \node(2) [right=of 1] {$\dots$};
        \node(3) [right=of 2] {$\bullet$};
        \node(4) [right=of 3] {$\bullet$};
        \node(5) [right=of 4] {$\dots$};
        \node(6) [right=of 5] {$\bullet$};
        \node(7) [right=of 6] {$\bullet$};

        \node[above=1pt of 0] {$x$};
        \node[above=1pt of 1] {$y_1$};
        \node[above=1pt of 3] {$y_m$};
        \node[above=1pt of 4] {$b_1$};
        \node[above=1pt of 6] {$b_n$};
        \node[above=1pt of 7] {$a$};

        \draw[->] (0) to (1);
        \draw[->] (1) to (2);
        \draw[->] (2) to (3);
        \draw[->] (3) to (4);
        \draw[->] (4) to (5);
        \draw[->] (5) to (6);
        \draw[->] (6) to (7);
    \end{tikzpicture}
\end{figure}

By \cref{pushout-char}, the existence of the edge $y_m \to b_1$ implies the existence of an edge $y_m \to \overline{b}_1$ in $X$, with $\overline{b}_1 \in A_V$, such that $f(\overline{b}_1) = b_1$. By the fact that $\pi$ is a projecting decomposition, it follows that there is a path of minimal length from $x$ to $\overline{b}_1$ in $X$ which passes through $\pi x$. In fact, the assumption that $p$ is of minimal length, together with \cref{pushout-char}, implies that the minimal length of a path from $x$ to $\overline{b}_1$ in $X$ is precisely $m+1$, as any shorter path in $X$ would induce a shorter path in $X'$. Thus we may write this path from $x$ to $\overline{b}_1$ as:

\begin{figure}[H]
    \centering
    \begin{tikzpicture}[node distance=20pt]
        \node(0) {$\bullet$};
        \node(1) [right=of 0] {$\bullet$};
        \node(2) [right=of 1] {$\dots$};
        \node(3) [right=of 2] {$\bullet$};
        \node(4) [right=of 3] {$\bullet$};
        \node(5) [right=of 4] {$\bullet$};
        \node(6) [right=of 5] {$\dots$};
        \node(7) [right=of 6] {$\bullet$};
        \node(8) [right=of 7] {$\bullet$};

        \node[above=1pt of 0] {$x$};
        \node[above=1pt of 1] {$z_1$};
        \node[above=1pt of 3] {$z_k$};
        \node[above=1pt of 4] {$\pi x$};
        \node[above=1pt of 5] {$c_1$};
        \node[above=1pt of 7] {$c_l$};
        \node[above=1pt of 8] {$\overline{b}_1$};        
        
        \draw[->] (0) to (1);
        \draw[->] (1) to (2);
        \draw[->] (2) to (3);
        \draw[->] (3) to (4);
        \draw[->] (4) to (5);
        \draw[->] (5) to (6);
        \draw[->] (6) to (7);
        \draw[->] (7) to (8);        
    \end{tikzpicture}
\end{figure}

where $z_i \in X_V \setminus A_V$ for all $i$, $c_j \in A_V$ for all $j$, and $k+l+1 = m$. Thus we obtain the following path in $X'$:

\begin{figure}[H]
    \centering
    \begin{tikzpicture}[node distance=20pt]
        \node(0) {$\bullet$};
        \node(1) [right=of 0] {$\bullet$};
        \node(2) [right=of 1] {$\dots$};
        \node(3) [right=of 2] {$\bullet$};
        \node(4) [right=of 3] {$\bullet$};
        \node(5) [right=of 4] {$\bullet$};
        \node(6) [right=of 5] {$\dots$};
        \node(7) [right=of 6] {$\bullet$};
        \node(8) [right=of 7] {$\bullet$};
        \node(9) [right=of 8] {$\dots$};
        \node(A) [right=of 9] {$\bullet$};
        \node(B) [right=of A] {$\bullet$};
        
        \node[above=1pt of 0] {$x$};
        \node[above=1pt of 1] {$z_1$};
        \node[above=1pt of 3] {$z_k$};
        \node[above=1pt of 4] {$f(\pi x)$};
        \node[above=1pt of 5] {$f(c_1)$};
        \node[above=1pt of 7] {$f(c_l)$};
        \node[above=1pt of 8] {$\phantom{b_1=}f(\overline{b}_1) = b_1$};
        \node[above=1pt of A] {$b_n$};        
        \node[above=1pt of B] {$a$};        
        
        \draw[->] (0) to (1);
        \draw[->] (1) to (2);
        \draw[->] (2) to (3);
        \draw[->] (3) to (4);
        \draw[->] (4) to (5);
        \draw[->] (5) to (6);
        \draw[->] (6) to (7);
        \draw[->] (7) to (8);        
        \draw[->] (8) to (9);        
        \draw[->] (9) to (A);        
        \draw[->] (A) to (B);        
    \end{tikzpicture}
\end{figure}

The length of this path is at most $k + l + n + 2 = m + n + 1$, the length of the original path $p$. As $p$ was assumed to be of minimal length, we have thus constructed a path of minimal length from $x$ to $a$ passing through $\pi' x = f(\pi x)$.
\end{proof}

As an aside, we mention that our cofibrations are not stable under the operation of taking \emph{pushout-(box) products}, meaning that the cofibration category structure of \cref{cofib-cat} is not monoidal with respect to the box product.
(Since this is only a comment, we do not recall the definition of a monoidal cofibration category in full detail.)
This sets our cofibration category of directed graphs apart from many familiar cofibration category structures, e.g., the Serre cofibration category structure on topological spaces of \cref{ex:top-cofib-cat}, which is monoidal with respect to the cartesian product.

\begin{proposition}
Let $A \cto X$ and $B \cto Y$ be cofibrations such that $X \setminus A$ contains a vertex admitting a path to $A$, and likewise $Y \setminus B$ contains a vertex admitting a path to $B$.
Then the pushout box product $X \gtimes B \cup_{A \gtimes B} A \gtimes Y \cto X \gtimes Y$ is not a cofibration.
\end{proposition}

\begin{proof}
By assumption, we have a vertex $x \in X_V$, not contained in $A$, admitting an edge to a vertex $a \in A_V$, and likewise we have a vertex $y \in Y_V$, not contained in $B$, admitting an edge to a vertex $b \in B_V$. The vertex $(x,y)$ of $X \gtimes Y$ is not contained in $X \gtimes B \cup_{A \gtimes B} A \gtimes Y$. However, this vertex admits edges to the two distinct vertices $(a,y), (x,b)$ of $X \gtimes B \cup_{A \gtimes B} A \gtimes Y$. By \cref{pi-closest}, it follows that $A \gtimes Y$ does not admit a projecting decomposition with respect to $X \gtimes B \cup_{A \gtimes B} A \gtimes Y$.
\end{proof}

We next prove that our cofibrations, like the cofibrations of a model category, are closed under retracts and transfinite composition.

\begin{proposition}\label{retract-closure}
Let $A \cto X$ be a cofibration and let $B \to Y$
be a retract of $A \cto X$, as depicted below:
\[
\begin{tikzcd}
B \arrow[d] \arrow[r,"i",above] & A \arrow[d] \arrow[r,"f",above] & B \arrow[d] \\
Y \arrow[r,"j",above] & X \arrow[r,"g",above] & Y \\
\end{tikzcd}
\]
where $fi = \id[B]$ and $gj = \id[Y]$. Then $B \to Y$ is a cofibration.
\end{proposition}

\begin{proof}
The fact that $A \cto X$ is a cofibration implies that all maps in the left-hand square are inclusions on vertices, so we will consider vertices in $B$ as vertices in $A,Y,$ or $X$ without relabeling. 

We first show that $B \to Y$ is an induced subgraph inclusion. Consider vertices $b, b' \in B_V$ such that there is an edge $b \to b'$ in $Y$. Then there is an edge $b \to b'$ in $X$, hence also in $A$ as $A$ is an induced subgraph of $X$. Taking the image of this edge under the retraction $f \colon A \to B$ gives an edge $b \to b'$ in $B$.

A similar argument shows that there are no edges out of $B$ in $Y$. Explicitly, consider a pair of vertices $b \in B_V, y \in Y_V$ with an edge $b \to y$. Then $b \in A$, implying that $y \in A$ as well since there are no edges out of $A$ in $X$. It follows that $g(y) = y$ is contained in $B$, as $f \colon A \to B$ is the restriction of $g$ to $A$. 

Finally, we will show that $Y$ admits a projecting decomposition with respect to $B$. Let $\pi$ denote the projecting decomposition of $X$ with respect to $A$; we will show that $\pi |_{Y_V^B}$ is a projecting decomposition of $Y$ with respect to $B$.

Let $y$ be a vertex of $Y$ admitting a path to some vertex $b$ of $B$. We first note that by \cref{pi-closest}, $\pi y$ is the unique vertex of $A$ admitting a path from $y$ of length $h(y)$ in $X$. The image of this path under $g$ defines a path from $y$ to $f \pi y$ in $Y \subseteq X$ of length less than or equal to $h(y)$; by the minimality of $h(y)$ and uniqueness of $\pi y$, it follows that $f \pi y = \pi y$. This in turn implies that $\pi y \in B$.

Now observe that by the definition of a projecting decomposition, there is a path in $X$ from $y$ to $b$, of minimal length among such paths, which passes through $\pi y$. Denote the length of this path by $n$; then the image of this path under $g$ defines a path from $y$ to $b$ in $Y$, passing through $f \pi y$, of length less than or equal to $n$ (hence equal to $n$ by minimality). Moreover, as $n$ is the minimum length of a path in $X$ from $y$ to $b$, it is likewise the minimum length of such paths in $Y \subseteq X$. Thus $\pi |_{Y}$ is indeed a projecting decomposition of $Y$ with respect to $B$.
\end{proof}

\begin{proposition}\label{closed-under-transf-comp}
The class of cofibrations is closed under transfinite composition.
\end{proposition}

\begin{proof}
Let $\alpha$ be an arbitrary limit ordinal. Consider a diagram $X \colon \alpha \to \DiGraph$ as depicted below:
\[
X_0 \to X_1 \to \ldots
\]
in which each map $X_\beta \cto X_{\beta+1}$
is a cofibration. For every limit ordinal $\beta < \alpha$, let $X_\beta$ denote the union of all $X_\gamma$ for $\gamma < \beta$, i.e.~ the colimit of the restricted diagram $X|_{\beta}$. Let $X_\alpha$ denote the colimit of this diagram, so that the map $X_0 \to X_\alpha$
is the transfinite composite of the maps $X_{\beta} \cto X_{\beta + 1}$. We will show, by induction on $\beta \leq \alpha$, that for each $\gamma < \beta$ the map $X_\gamma  \cto X_\beta$ is a cofibration; in particular, this implies that $X_0 \cto X_\alpha$ is a cofibration.

In the base case $\beta = 0$, there is no $\gamma < \beta$, so the statement is vacuously true.

Let $\beta > 0$ and suppose the statement holds for all $\beta' < \beta$. In the case of a successor ordinal $\beta = \beta' + 1$, the transfinite composite factors as $X_\gamma \to X_{\beta'} \cto X_{\beta' + 1}$.
If $\gamma = \beta'$ then the map $X_\gamma \cto X_{\beta'}$ is the identity, while if $\gamma < \beta'$ then $X_\gamma \cto X_{\beta'}$ is a cofibration by the induction hypothesis. The map $X_{\beta'} \cto X_{\beta' + 1}$ is a cofibration by assumption and so the composite is a cofibration by \cref{cof-wide-subcat}.

Now we consider the case of a limit ordinal $\beta$, so that $X_\beta$ is the union of all the directed graphs $X_{\beta'}$ for $\beta' < \beta$. We first show that $X_\gamma \to X_\beta$ 
is an induced subgraph inclusion. Let $x, x'$ be a pair of vertices of $X_\gamma$ such that $X_\beta$ contains an edge $x \to x'$. Then there is an edge $x \to x'$ in $X_{\beta'}$ for some $\gamma \leq \beta' < \beta$. Since $X_\gamma \cto X_{\beta'}$ is a cofibration by the induction hypothesis, and hence an induced subgraph inclusion, there is an edge $x \to x'$ in $X_\gamma$ as well.

Next let $x$ be a vertex of $X_\gamma$ and $x'$ a vertex of $X_\beta$ not contained in $X_\gamma$; we will show there is no edge $x' \to x$ in $X_\beta$. Note that $x'$ is a vertex of $X_{\beta'}$ for some $\beta' < \beta$, and $X_\gamma \cto X_{\beta'}$ is a cofibration by the induction hypothesis. Thus there is no edge $x' \to x$ in $X_{\beta'}$; since we have just proven that $X_{\beta'} \to X_\beta$ 
is an induced subgraph inclusion, the same is true in $X_\beta$.

Finally, we show that $X_\beta$ admits a projecting decomposition $\pi$ with respect to $X_\gamma$. By \cref{rmk:projecting_alternative}, it suffices to define $\pi$ satisfying satisfying the criteria of \cref{proj-decomp-def} on $x \in (X_\beta)^{X_{\gamma}}_V \setminus (X_{\gamma})_V$. Let $\beta'$ be minimal such that $x\in X_{\beta'}$. We must have that $\gamma < \beta' < \beta$. By the induction hypothesis, the inclusion $X_\gamma \cto X_{\beta'}$ is a cofibration and thus admits a projecting decomposition $\pi_{\beta'}$. We define $\pi x = \pi_{\beta'} x$.

Note that any path in $X_\beta$ from $x \in (X_{\beta'})_V$ to $y\in (X_\gamma)_V$ must be contained in $X_{\beta'}$, since there are no edges from $X_{\beta'}$ to $ X_{\beta}$. Therefore a minimal length path from $x$ to $y$ in $X_{\beta'}$ which passes through $\pi_{\beta'}x$ is also of minimal length in $X_\beta$ and passes through $\pi x=\pi_{\beta'}x$.
\end{proof}

In view of the results above, given a set of cofibrations, we may consider the class of cofibrations which it generates under pushout, retract, and transfinite composition. 
We might naturally hope to find some set of cofibrations which generates the entire class of cofibrations in $\DiGraph$ via the procedure outlined above; however, \cref{no-gen-cof-set} below shows that this is not possible.
Since we are not working with a model category structure, we must first define our concept of generation under these operations precisely.

\begin{definition}
Let $C$ denote a set of cofibrations in $\DiGraph$. For each ordinal $\alpha$, we define a class of cofibrations $\Gen_\alpha(C)$ by transfinite induction on $\alpha$ as follows.
\begin{itemize}
    \item $\Gen_0(C) = C$;
    \item for a successor ordinal $\alpha = \beta + 1$, $\Gen_\alpha(C)$ is the class of all retracts, pushouts, and transfinite composites of maps in $\Gen_\beta(C)$. In particular, this implies that $\Gen_\beta(C)$ is contained in $\Gen_\alpha(C)$;
    \item for a limit ordinal $\alpha$, $\Gen_\alpha(C)$ is the class of all maps contained in some $\Gen_\beta(C)$ for $\beta < \alpha$.
\end{itemize}
Note that in general, $\Gen_\alpha(C)$ will be a proper class for $\alpha > 0$.

We then define $\Gen(C)$, the \emph{class of cofibrations generated by $C$}, to be the union of all the classes $\Gen_{\alpha}(C)$, \ie{} the class of all maps contained in $\Gen_\alpha(C)$ for some $\alpha$.
\end{definition}

\begin{proposition}\label{no-gen-cof-set}
There is no set of cofibrations $C$ such that $\Gen(C)$ is the class of all cofibrations.
\end{proposition}

\begin{proof}
Let $C$ denote an arbitrary set of cofibrations. Let  $S$ denote a set whose cardinality is greater than that of the vertex set of the codomain of any map in $C$ and let $K_S$ denote the complete directed graph having $S$ as its set of vertices. Given $A \cto X$, let $X - A$ denote the induced subgraph on $X_V\setminus A_V.$

By transfinite induction on ordinals $\alpha$, we will show that $\Gen_\alpha(C)$ does not contain any cofibration $A \cto X$ such that $X - A$ contains $K_S$ as a subgraph. (For instance, we may consider $A = \varnothing, X = K_S$.) 

In the base case $\alpha = 0$, this is immediate from the definition of $S$.

Now suppose that the statement holds for all $\beta < \alpha$. If $\alpha$ is a limit ordinal, then it is immediate from the induction hypothesis that the statement holds for $\alpha$ as well. Now consider the case of a successor ordinal $\alpha = \beta + 1$. Given a cofibration $A \cto X$ with an induced subgraph inclusion $K_S \subseteq X-A$, to show that $A \cto X$ is not contained in $\Gen_\alpha(C)$, we must show that it is not a retract, pushout, or transfinite composite of maps in $\Gen_\beta(C)$.

For the case of retracts, suppose that $A \cto X$ is a retract of some cofibration $B \cto Y$, as depicted below.
\[
\begin{tikzcd}
A \arrow[d] \arrow[r,"i",above] & B \arrow[d] \arrow[r,"f",above] & A \arrow[d] \\
X \arrow[r,"j",above] & Y \arrow[r,"g",above] & X \\
\end{tikzcd}
\]
Then we have a subgraph inclusion $K_S \subseteq X \subseteq Y$. If some vertex $s$ of $K_S$ is contained in $B$, then $f(s) = s$ is contained in $A$. This contradicts our assumption that $K_S$ is contained entirely in $X \setminus A$. Thus we see that $K_S$ is contained in the complement $Y \setminus B$. By the induction hypothesis, it follows that $B \cto Y$ is not contained in $\Gen_\beta(C)$.

Next suppose that $A \cto X$ is a pushout of some cofibration $B \cto Y$, as depicted below.
\[
\begin{tikzcd}
B \arrow[r] \arrow[d,tail] \pushout & A \arrow[d,tail] \\
Y \arrow[r] & X \\
\end{tikzcd}
\]
By assumption, no vertex of $K_S$ is contained in $A$, so $K_S$ is contained entirely in the image of $Y \to X$. Thus, by \cref{pushout-char}, $Y$ contains $K_S$ as a subgraph. Moreover, no vertex of $K_S \subseteq Y$ can be contained in $B$, as this would imply that the corresponding vertex of $X$ was contained in $A$ by the commutativity of the diagram. It follows that $K_S$ is a subgraph of $Y \setminus B$, implying that $B \cto Y$ is not contained in $\Gen_\beta(C)$ by the induction hypothesis.

Now suppose that $A \cto X$ is a transfinite composite of some family of cofibrations $X_\gamma \to X_{\gamma + 1}$, indexed by $\gamma < \delta$ for some ordinal $\delta$, with $X_0 = A$. Then for each such $\gamma$ we have $X_\gamma \subseteq X$, and $X$ is the union of all the $X_\gamma$; in particular, each vertex of $X$ is contained in some subgraph $X_\gamma$. Moreover, for $\rho < \gamma$ there are no edges from the vertices of $X_\gamma \setminus X_\rho$ to those of $X_\rho$, as the map $X_\rho \to X_\gamma$ is a cofibration by \cref{closed-under-transf-comp}. 

Let $s$ denote an arbitrary vertex of $K_S \subseteq X$, and let $\gamma$ be minimal such that $s$ is a vertex of $X_\gamma$. If $\gamma$ were a limit ordinal, then $X_\gamma$ would be the union of all $X_\rho$ for $\rho < \gamma$, contradicting minimality. Thus $\gamma = \gamma' + 1$ for some $\gamma'$ such that $s$ is not contained in $X_{\gamma'}$. As every other vertex of $K_S$ admits edges to and from $s$, all vertices of $K_S$ must also appear in $X_\gamma$, and not in $X_{\gamma'}$. As $X_{\gamma} \cto X$ is an induced subgraph inclusion, it follows that $K_S$ is a subgraph of $X_\gamma \setminus X_{\gamma'}$. By the induction hypothesis, the cofibration $X_{\gamma'} \cto X_\gamma$ is not contained in $\Gen_{\beta}(C)$.

Thus we see that $A \cto X$ is not a retract, pushout or transfinite composite of any map in $\Gen_{\beta}(C)$, so it is not contained in $\Gen_{\alpha}(C)$.

By induction, no cofibration $A \cto X$ such that $K_S$ is a subgraph of $X \setminus A$ is contained in $\Gen(C)$. Thus $\Gen(C)$ is not the class of all cofibrations.
\end{proof}


\section{Excision} \label{sec:excision}

In the present section, we formulate and prove the excision axiom for path homology in terms of the cofibrations of \cref{cofib-def}.
This plays a crucial role in \cref{sec:main-thm} when we establish a cofibration category structure on $\DiGraph$ (\cref{cofib-cat}).

We now turn our attention to the \emph{relative homology modules} $H_n(X,A)$ associated to a cofibration $A \cto X$.
Throughout this section, we fix a commutative coefficient ring $R$.

As our analysis of these groups will involve several constructions which are functorial with respect to commuting squares of cofibrations, we define $\Cof$ to be the category with cofibrations in $\DiGraph$ as its objects and commuting squares as its morphisms. Let $\CofPO$ denote the subcategory of $\Cof$ which contains all objects, but with only pushout squares as morphisms. Throughout this section, a diagram of the form
\[
\begin{tikzcd}
A \arrow[r,"f|_A"] \arrow[d,tail] & B \arrow[d,tail] \\
X \arrow[r,"f"] & Y \\
\end{tikzcd}
\]
\ie{} a morphism in $\Cof$ or $\CofPO$, will be denoted by $f \colon (X,A) \to (Y,B)$.

\subsection*{Statement of excision axiom}

The goal of this short subsection is to give a statement of the `excision axiom' (\cref{pushout-rel-hom}), namely that relative homology takes homotopy pushouts to isomorphisms.

\begin{definition}
Given a subgraph inclusion $A \subseteq X$, the $n^{\mathrm{th}}$ \emph{relative homology module} $H_{n}(X,A)$ is the $n^{\mathrm{th}}$ homology module of the factor complex $\Omega(X) / \Omega(A)$.
\end{definition}

Given a cofibration $A \cto X$, we let $X - A$ denote the complement of $A$ in $X$, \ie{} the induced subgraph on the vertices of $X$ not contained in $A$. We define $(X - A)^A$ to be the induced subgraph of $X - A$ on the vertices which admit paths to $A$. We let $\W$ denote the induced subgraph of $X$ on the vertices of height 1, \ie{} those vertices not in $A$ which admit edges into $A$. 


\begin{definition}
Let $Q \colon \Cof \to \Ch_R$ denote the functor which sends a cofibration $A \cto X$ to the factor complex $\Omega(X)/\Omega(A)$, and a morphism of cofibrations to the induced map between their factor complexes. 
\end{definition}

We observe that the composite of $Q$ with the homology functor $H_* \colon \Ch_R \to \ModRN$ sends each cofibration to its family of relative homology modules, and each commuting diagram to its family of induced maps on relative homology modules.

\begin{proposition}[{\cite[Thm.~3.11]{grigor'yan-jimenez-muranov-yau}}]\label{rhom-les}
For any subgraph inclusion $A \hookrightarrow X$, there is a relative homology long exact sequence:
\[
\cdots \to H_n(X) \to H_{n}(X,A) \to H_{n-1}(A) \to \cdots \to H_0(A) \to H_0(X) \to H_0(X,A) \to 0\text{.}
\]
\end{proposition}

\begin{corollary}\label{tcof-rel-hom}
A subgraph inclusion $A \hookrightarrow X$ is a homology isomorphism if and only if all relative homology modules $H_n(X,A)$ are zero. \qed
\end{corollary}

Our main goal for this section is to prove the following.

\begin{theorem}\label{pushout-rel-hom}
Given a pushout diagram $f \colon (X,A) \to (X',A')$ in $\CofPO$,
the induced map of relative homology modules $H_n(X,A) \to H_n(X',A')$ is an isomorphism for all $n \geq 0$.
\end{theorem}

In \cref{sec:main-thm}, we will use this result to prove that cofibrations and homology isomorphisms form a cofibration category structure on $\DiGraph$.

\subsection*{The complement chain complex}

Note that \cref{pushout-rel-hom} is equivalent to the statement that the composite functor $H_* Q \colon \Cof \to \ModRN$ sends all pushout squares to isomorphisms of $\ModRN$. Our strategy for proving the latter is as follows: we first define the `complement chain complex' functor $\tOmega \colon \Cof \to \Ch_R$, and establish a natural isomorphism $Q \cong \tOmega$.
In the next subsection, we define another functor $M \colon \CofPO \to \Ch_R$, show that $M$ sends all morphisms of $\CofPO$ (\ie{} all pushout squares) to isomorphisms of $\Ch_R$, and establish a natural isomorphism $\tOmega|_{\CofPO} \cong M$.


We begin by defining the intermediate functor $\tOmega$.

\begin{definition}
Let $A \hookrightarrow X$ be an induced subgraph inclusion. For $n \geq 0$, the $R$-module  $\tA_n(X,A)$ is the submodule of $A_n(X)$ generated by the paths which intersect the complement $X - A$. The $R$-module $\tOmega_n(X,A)$ is defined as the intersection $\Omega_n(X) \cap \tA_n(X,A)$. The $R$-modules $\tOmega_n(X)$ assemble to a chain complex $\tOmega(X,A)$ as follows:  for $\omega \in \tOmega_n(X,A)$, the boundary $\bd \omega$ is computed by first computing the boundary of $\omega$ as an element of $\Omega_n(X)$, then setting any terms corresponding to paths not intersecting $X - A$ to $0$.
\end{definition}



\begin{lemma}\label{Omega-U-sum}
Let $A \cto X$ be a cofibration of directed graphs. Then for each $n \geq 0$, the inclusion $\Omega_n(A) \hookrightarrow \Omega_n(X)$ is isomorphic to the direct summand inclusion $\Omega_n(A) \hookrightarrow \Omega_n(A) \oplus \tOmega_n(X,A)$.
\end{lemma}

\begin{proof}
This statement is essentially a strengthening of (the dual of) \cite[Lem.~3.10]{grigor'yan-jimenez-muranov-yau}. The proof supplied in that reference shows that any linear combination $\omega \in \Omega_n(X)$ decomposes uniquely as $p(\omega) + q(\omega)$, where $p(\omega) \in \Omega_n(A)$ and $q(\omega) \in \tOmega_n(X,A)$, and the maps $p$ and $q$ thus obtained are $R$-module homomorphisms. Thus we obtain an isomorphism $\Omega_n(X) \to \Omega_n(A) \oplus \tOmega_n(X,A)$ sending each $\omega$ to $(p(\omega),q(\omega))$. The pre-image of the summand $\Omega_n(A)$ under this isomorphism is precisely $\Omega_n(A)$ viewed as a submodule of $\Omega_n(X)$; it thus follows that the inclusions $\Omega_n(A) \hookrightarrow \Omega_n(X)$ and $\Omega_n(A) \hookrightarrow \Omega_n(A) \oplus \tOmega_n(X,A)$ are isomorphic.
\end{proof}

\begin{corollary} \label{Omega-U-iso}
Given a cofibration of directed graphs $A \cto X$, for $n \geq 0$, the quotient module $\Omega_n(X) / \Omega_n(A)$ is isomorphic to $\tOmega_n(X,A)$. Moreover, these isomorphisms define an isomorphism of chain complexes $\Omega(X) / \Omega(A) \cong \tOmega(X,A)$.
\end{corollary}

\begin{proof}
The stated isomorphisms of $R$-modules are immediate from \cref{Omega-U-sum}. That these isomorphisms commute with the boundary operations of $\Omega(X) / \Omega(A)$ and $\tOmega(X,A)$ follows from a routine calculation.
\end{proof}

\begin{remark}
Note that the proofs of \cref{Omega-U-sum,Omega-U-iso} do not require the condition that $A \cto X$ admits a projecting decomposition, only that there are no edges out of $A$ in $X$. Indeed, \cite[Lem.~3.10]{grigor'yan-jimenez-muranov-yau} assumes only (the dual of) the latter condition.
\end{remark}

\begin{definition}\label{tOmega-def}
We define a functor $\tOmega \colon \Cof \to \Ch_R$ as follows. Given a cofibration $A \cto X$, its image under $\tOmega$ is the chain complex $\tOmega(X,A)$. Given a commuting square $f \colon (X,A) \to (Y,B)$, each $R$-module homomorphism $\tOmega_n(f) \colon \tOmega_n(X,A) \to \tOmega_n(Y,B)$ is the composite of the restriction of $\Omega_n(f) \colon \Omega_n(X) \to \Omega_n(Y)$ with the quotient map $\Omega_n(Y) \cong \Omega_n(B) \oplus \tOmega_n(Y,B) \to \tOmega_n(Y,B)$ which sends the summand $\Omega_n(B)$ to $0$.
\end{definition}

\begin{proposition}\label{Omega-U}
The functors $Q, \tOmega \colon \Cof \to \Ch_R$ are naturally isomorphic.
\end{proposition}

\begin{proof}
Given a cofibration $X \cto A$, its images under $Q$ and $\tOmega$ are isomorphic by \cref{Omega-U-iso}. To see that this isomorphism is natural, consider a morphism $f \colon (X,A) \to (Y,B)$ in $\Cof$. Given $\omega \in \tOmega_n(X,A)$, the composites $\tOmega_n(X,A) \cong \Omega_x(X) / \Omega_n(A) \xrightarrow{Q f} \Omega_n(Y) / \Omega_n(B)$ and $\tOmega_n(X,A) \xrightarrow{\tOmega(f)} \tOmega_n(Y,B) \cong \Omega_n(Y) / \Omega_n(B)$ both send $\omega$ to the equivalence class of $f (\omega)$ modulo $\Omega_n(B)$ in $\Omega_n(Y)$.
\end{proof}

We next construct the functor $M \colon \CofPO \to \Ch_R$.

\begin{definition}
Let $A \cto X$ be a cofibration. For $n \geq 0$, let $\tAOne_n(X,A)$ denote the submodule of $A_n(X - A)$ generated by the paths whose last vertices are in the subgraph $\W$. The $R$-module $\tOmegaOne_{n}(X,A)$ is defined by the following pullback diagram:
\[
\begin{tikzcd}
\tOmegaOne_n(X,A) \arrow[r] \arrow[hook,d] \pullback & \tAOne_{n-1}(X,A) \arrow[hook,d] \\
\tAOne_n(X,A) \ar[r,"\bd",above] & C_{n-1}(X - A) \\ 
\end{tikzcd}
\]
In other words, $\tOmegaOne_n(X,A)$ is the submodule of $\Omega_{n}(X)$ consisting of linear combinations of allowed paths with their last vertices in $\W$, whose boundaries again have their last vertices in $\W$. Similar to the definition of $\tOmega(X,A)$, the boundary map on $C_n(X)$ restricts to a map $\tOmegaOne_n(X,A) \to \tOmegaOne_{n-1}(X,A)$ for each $n$; we thus obtain a chain complex $\tOmegaOne(X,A) \subseteq \tOmega(X,A)$.

Now consider a morphism in $\CofPO$, i.e. a pushout diagram in which the vertical maps are cofibrations, as depicted below:
\[
\begin{tikzcd}
    A \arrow[tail,d] \arrow[r,"f|_A"] \pushout & B \arrow[tail,d]  \\
    X \arrow[r,"f"] & Y \\
\end{tikzcd}
\]
By \cref{pushout-complement}, for each $n$ the map $C_n(X) \to C_n(Y)$ restricts to an isomorphism $C_n(X - A) \cong C_n(Y - B)$, and these in turn restrict to isomorphisms $\tAOne_n(X,A) \to \tAOne(Y,B)$ which commute with boundaries. Thus we have a natural isomorphism of cospans:
\[
\begin{tikzcd}
    \tAOne_{n-1}(X,A) \arrow[d,"\cong"] \arrow[hook,r] & C_{n-1}(X-A) \arrow[d,"\cong"] & \tAOne_n(X,A) \arrow["\bd",swap,l] \arrow[d,"\cong"] \\
    \tAOne_{n-1}(Y,B) \arrow[hook,r] & C_{n-1}(Y-B) & \tAOne_n(Y,B) \arrow["\bd",swap,l]  \\
\end{tikzcd}
\]
This yields a natural isomorphism between pullbacks $\tOmegaOne_n(X,A) \cong \tOmegaOne(Y,B)$. As this map commutes with boundaries by construction, it defines a chain complex isomorphism $\tOmegaOne(X,A) \cong \tOmegaOne(Y,B)$.

We thus define a functor $\tOmegaOne \colon \CofPO \to \Ch_R^\to$ sending each object $A \cto X$ to the chain complex inclusion $\tOmegaOne(X,A) \hookrightarrow \Omega(X - A)$ and each morphism $(X,A) \to (Y,B)$ of $\CofPO$ to the commuting square below:
\[
\begin{tikzcd}
\tOmegaOne(X,A) \arrow[r,"\cong"] \arrow[d,hook] & \tOmegaOne(Y,B) \arrow[d,hook] \\
\Omega(X - A) \arrow[r,"\cong"] & \Omega(Y - B) \\
\end{tikzcd}
\]
It is straightforward to verify that this construction is functorial.
\end{definition}

\subsection*{Mapping cone complex of a cofibration}

\begin{definition}\label{def:M}
Let $M \colon \CofPO \to \Ch_R$ denote the composite of the functor $\tOmegaOne,$ defined above, with the mapping cone functor $\Ch^{\rightarrow}_R \to \Ch_R$. 
We refer to $M(X, A)$ as the \emph{mapping cone complex} of $A \cto X$.
\end{definition}

More explicitly, the functor $M$ may be described as follows:
\begin{itemize}
    \item for $n \geq 0$, $M_n(X,A) = \tOmegaOne_{n-1}(X,A) \oplus \Omega_n(X-A)$;
    \item for $(p,q) \in M_n(X,A)$, $\bd(p,q) = (-\bd p, \bd q - p)$.
    \item given a pushout square $f \colon (X,A) \to (X',A')$, $Mf$ acts in each degree $n$ as the direct sum of the isomorphisms $\tOmegaOne_{n-1}(X,A) \cong \tOmegaOne_{n-1}(X',A')$ and $\Omega_n(X-A) \cong \Omega_n(X' - A')$ induced by $f$.
\end{itemize}


The following result is then immediate.

\begin{lemma}\label{mapping-cone-invariant}
The functor $M \colon \CofPO \to \Ch_R$ sends all morphisms of $\CofPO$ to isomorphisms of $\Ch_R$. \qed
\end{lemma}

Our next goal is to establish a natural isomorphism $E$ from $M$ to the restriction of $\tOmega$ to $\CofPO$.
(For ease of notation, we also write $\tOmega$ for the restriction of $\tOmega \colon \Cof \to \Ch_R$ to $\CofPO$, relying on context to remove ambiguity.) 

The definition of $E$ is fairly involved and will only be given in \cref{iso-def}, since its well-definedness (\cref{iso-pres-bdry}) depends on preceding results.
To provide a roadmap, let us briefly summarize the strategy.
We begin by defining linear maps $L^j$ acting on chain complexes associated to the cofibration $A \cto X$.
Although at that point we could state the definition of $E$, since the required formula only involves the function $L^0$, we defer the definition, by first proving a sequence of lemmas (\cref{L-injective,L-truncate-pi,L-truncate-minimal,L-allowed,pi-allowed,L-bdry,L-Omega-to-Omega}) that establish well-definedness of $E.$

From here on, let $A \cto X$ denote an arbitrary cofibration. We will regularly write an arbitrary path in $X$ as $x = x_0 \cdots x_n$ and will write $a_i$ for the vertex $\pi x_i \in A_V$.

\begin{definition}\label{L-def}
For $n \geq 1$, we define a family of linear maps $L^j \colon C_{n-1}((X - A)^A) \to C_n(X)$ for $0 \leq j \leq  n-1$. Given a generator $x = x_0 \cdots x_{n-1} \in C_{n-1}((X - A)^A)$, let
\[
L^j(x) = \sum_{i = j}^{n-1} (-1)^i x_0 \cdots x_i a_i \cdots a_{n-1}.
\]
This definition extends by linearity to define each $L^j$ on all of $C_{n-1}((X - A)^A)$.
\end{definition}


As the definition above may appear very technical, we will briefly discuss some of the intuition behind it. Suppose we are given a path $x = x_0 \cdots x_{n-1}$ lying entirely in $\W$, \ie{} such that $h(x_i) = 1$ for all $i$. Then by \cref{pi-closest,pi-edges}, the graph $A$ contains a grid of squares formed by the path $x$, its projection to $A$, and the edges $x_i \to a_i$, pictured below for a path of length $4$.

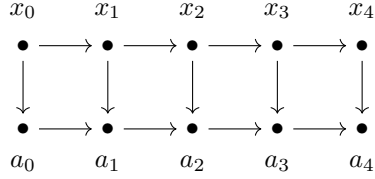
\begin{figure}[H]
    \centering
    \begin{tikzpicture}[node distance=20pt]
        \node(0) {$\bullet$};
        \node(1) [right=of 0] {$\bullet$};
        \node(2) [right=of 1] {$\bullet$};
        \node(3) [right=of 2] {$\bullet$};
        \node(4) [right=of 3] {$\bullet$};
        \node(5) [below=of 0] {$\bullet$};
        \node(6) [right=of 5] {$\bullet$};
        \node(7) [right=of 6] {$\bullet$};
        \node(8) [right=of 7] {$\bullet$};
        \node(9) [right=of 8] {$\bullet$};
        
        \node[above=1pt of 0] {$x_0$};
        \node[above=1pt of 1] {$x_1$};
        \node[above=1pt of 2] {$x_2$};
        \node[above=1pt of 3] {$x_3$};
        \node[above=1pt of 4] {$x_4$};
        \node[below=1pt of 5] {$a_0$};
        \node[below=1pt of 6] {$a_1$};
        \node[below=1pt of 7] {$a_2$};        
        \node[below=1pt of 8] {$a_3$};  
        \node[below=1pt of 9] {$a_4$};                
        
        \draw[->] (0) to (1);
        \draw[->] (1) to (2);
        \draw[->] (2) to (3);
        \draw[->] (3) to (4);

        \draw[->] (5) to (6);
        \draw[->] (6) to (7);
        \draw[->] (7) to (8);        
        \draw[->] (8) to (9);    
        
        \draw[->] (0) to (5);        
        \draw[->] (1) to (6);        
        \draw[->] (2) to (7);        
        \draw[->] (3) to (8);    
        \draw[->] (4) to (9);        
    \end{tikzpicture}
\end{figure}

The element $L^0(x) \in C_n(X)$ is the alternating sum of all paths from the upper left corner (i.e.~$x_0$) to the lower right corner (i.e.~$a_4$) of this grid. 
It can be seen as a generalization of the generator of $\Omega_2(C_{2,2})$ of \cref{ex:commuting_square}.
Our strategy for proving the desired isomorphism, roughly speaking, involves showing that the cofibration conditions are sufficiently restrictive that any path in $X$ which forms part of a linear combination in $\tOmega_n(X,A)$ must either be contained entirely in $(X - A)$, or arise from a grid construction similar to the above (suitably generalized for paths not contained entirely in $\W$).

\begin{lemma}\label{L-injective}
For $n \geq 1$ and $0 \leq j \leq n-1$, the map $L^j \colon C_{n-1}((X - A)^A) \to C_{n}(X)$ is injective.
\end{lemma}

\begin{proof}
We first prove this result in the case $j = n-1$. Here we may note that for any generator (\ie{} any non-degenerate path) $x = x_0 \cdots x_{n-1}$ of $C_{n-1}((X - A)^A)$, we have $L^{n-1}(x) = (-1)^{n-1} x_0 \cdots x_{n-1} a_{n-1}$. This is necessarily non-degenerate: the assumption that $x$ is non-degenerate implies $x_i \neq x_{i+1}$ for $0 \leq i \leq n-2$, and $x_{n-1} \neq a_{n-1}$ since $x_{n-1}$ is not in $A$. Taking as a basis for $C_n(X)$ the set of elements $(-1)^{n-1} y$ for all non-degenerate $n$-paths $y$ in $X$, we thus see that distinct generators of $C_{n-1}((X-A)^A)$ are sent to distinct elements of this basis. It follows that $L^{n-1}$ is injective.

Now consider an arbitrary $0 \leq j \leq n-2$. Suppose that for some $p \in C_{n-1}((X - A)^A)$ we have $L^j(p) = 0$. We may rewrite this equation as $L^j(p) - L^{n-1}(p) = -L^{n-1}(p)$. Now note that for any generator $x$ as above, we have $L^j(x) - L^{n-1}(x) = \sum_{i = j}^{n-2} (-1)^ix_0 \cdots x_i a_i \cdots a_{n-1}$.  Thus each nonzero term of $L^j(p) - L^{n-1}(p)$ corresponds to a path including at least two vertices of $A$, while each nonzero term of $L^{n-1}(p)$ corresponds to a path including exactly one vertex of $A$. Thus this equality can only hold if both sides are equal to zero. In particular, we have $L^{n-1}(p) = 0$; as shown above, this implies $p = 0$.
\end{proof}

\begin{lemma}\label{L-truncate-pi}
For $n \geq 2$, let $x = x_0 \cdots x_{n-1}$ denote a generator of $\tAOne_{n-1}(X,A)$. Suppose that for some $0 \leq j \leq n-2$, we have $a_j = a_{j+1}$. Then $L^{j'}(x) = L^0(x)$ for all $j' \leq j$.
\end{lemma}

\begin{proof}
By definition, $L^0(x) = \sum\limits_{i = 0}^{n-1} (-1)^i x_0 \cdots x_i a_i \cdots a_{n-1}$. Thus we must show that all terms of this sum for which $i < j$ are zero. This follows from the fact that each such term includes the substring $a_j a_{j+1}$, and is therefore degenerate.
\end{proof}

\begin{lemma}\label{L-truncate-minimal}
For $n \geq 2$, let $x = x_0 \cdots x_{n-1}$ denote a generator of $\tAOne_{n-1}(X,A)$. Let $0 \leq j \leq n-1$ be minimal such that $h(x_i) = 1$ for all $i > j$. Then $L^0(x) = L^j(x)$.
\end{lemma}

\begin{proof}
If $j = 0$ then the statement is a tautology, so assume otherwise.  Furthermore, we may note that because $x$ is an allowed path with its last vertex in $\W$, we have $h(x_{n-1}) = 1$, so that $j < n - 1$.

By the minimality of $j$, we have $h(x_{j-1}) \neq 1$; as $x_{j-1}$ is not in $A$, it follows that $h(x_{j-1}) \geq 2$. As there is an edge $x_{j-1} \to x_j$, \cref{pi-edges} implies that $a_{j-1} = a_j$. The stated result thus follows from \cref{L-truncate-pi}.
\end{proof}

\begin{lemma}\label{L-allowed}
For all $n \geq 1$ and $p \in C_{n-1}((X-A)^A)$, we have $L^0(p) \in \tA_{n}(X,A)$
if and only if $p \in \tAOne_{n-1}(X,A)$.
\end{lemma}

\begin{proof}
We first show that if $p \in \tAOne_{n-1}(X,A)$ then $L^0(p) \in \tA_{n}(X,A)$.
It suffices to show that, for all generators $x \in \tAOne_{n}(X,A)$, $L^0(x)$ is allowed and intersects $X-A$. By \cref{L-truncate-minimal}, to show that $L^0(x)$ is allowed, it suffices to show that each path $x_1 \cdots x_i a_i \cdots a_{n-1}$ is allowed for $i \geq j$, where $j$ is minimal such that $h(x_i) = 1$ for all $i \geq j$. Therefore, let $x = x_0 \cdots x_{n-1}$ be an allowed path in $X$ with its last vertex in $A$, and consider the path
\[
x_0 \cdots x_i a_i \cdots a_{n-1}
\]
for some $i$ satisfying $j \leq i \leq n-1$. For each such $0 \leq k \leq i - 1$, there is an edge $x_k \to x_{k+1}$ by the assumption that $p$ is allowed. Similarly, for $i \leq k < n - 1$, the assumption that $p$ is allowed implies that there is an edge $x_k \to x_{k+1}$. Furthermore, the assumption that $j \leq i$ implies that $h(x_k) = h(x_{k+1}) = 1$; these two facts imply the existence of an edge $a_k \to a_{k+1}$ by \cref{pi-edges}. It remains to be shown that there is an edge $x_i \to a_i$; this is immediate from the assumption that $h(x_i) = 1$ and \cref{pi-closest}.

Now suppose that $L^0(p) \in \tA_{n}(X,A)$. We may write $p$ as a linear combination of generators $\sum_{k = 1}^m c_k x^k$, where each $c_k$ is a coefficient and each $x^k$ is an allowed path in $(X-A)^A$. Then the sum of all terms of $L^0(p)$ corresponding to paths which include exactly one vertex of $A$ is
\[
L^{n-1}(p) = \sum_{k=1}^{m} c_k (-1)^{n-1} x^k_0 \cdots x^k_{n-1} a^k_{n-1}
\]
where $a^k_{n-1} = \pi x^k_{n-1}$. Since the generators $x^k$ are distinct, the terms of this sum are distinct as well. Thus there are no cancellations among these terms. As all remaining terms of $L^0(p)$ correspond to paths including at least two vertices of $A$, none of these terms can cancel with those of the sum above either. Therefore, since  $L^0(p)$ is allowed, it must be the case that each path $x^k_0 \cdots x^k_{n-1} a_{n-1}$ is allowed. In particular, this implies that $x^k_0 \cdots x^k_{n-1}$ is allowed. Furthermore, the presence of an edge $x^k_{n-1} \to a^k_{n-1}$ implies that $h(x^k_{n-1}) = 1$. Thus $p \in \tAOne_{n-1}(X,A)$.
\end{proof}

\begin{definition}\label{pi-linear-def}
For $n \geq 0$, we define a linear map $\pi \colon C_{n}((X-A)^A) \to C_{n}(A)$ by sending each generator $x = x_0 \cdots x_n$ to $a_0 \cdots a_n$.
\end{definition}

\begin{lemma}\label{pi-allowed}
If $p \in \tAOne_{n}(X,A)$, then $\pi p \in A_n(A)$.
\end{lemma}

\begin{proof}
It suffices to consider the case where $p$ is a generator, i.e., an allowed path $x = x_0 \cdots x_n$; in this case we have $\pi x = a_0 \cdots a_n$. Our assumption that $x \in \tAOne_{n}(X,A)$ implies $h(x_n) = 1$, and that $h(x_i) \geq 1$ for all $i$, since $x_n$ is not in $A_V$ and there are no edges out of $A$. 

If $h(x_i) = 1$ for all $i$, then $\pi x$ is allowed by \cref{pi-edges}.(\ref{pi-edges-equal}). Now suppose that $h(x_j) > 1$ for some $j$; without loss of generality we may choose $j$ to be maximal, so that $h(x_{j+1}) = 1$. Then by \cref{pi-edges}.(\ref{pi-edges-greater}), we have $a_j = a_{j+1}$. Thus $\pi x$ is degenerate, and hence equal to zero.
\end{proof}

\begin{lemma}\label{L-bdry}
For $n \geq 1$ and $p \in \tAOne_{n-1}(X,A)$, we have $\bd L^0(p) = -L^0(\bd p) - p + \pi(p)$.
\end{lemma}

\begin{proof}
It suffices to show the given equality when $p$ is a generator $x = x_0 \cdots x_{n-1}$. In this case we compute:

\begin{align*}
    \bd L^0(x) & = \bd\left( \sum_{i = 0}^{n-1} (-1)^i x_0 \cdots x_i a_i \cdots a_{n-1}\right) \\ 
    & = \sum_{i = 0}^{n-1} (-1)^i \bd x_0 \cdots x_i a_i \cdots a_{n-1} \\  
    & = \sum_{i=0}^{n-1} (-1)^{i} (\sum_{j = 0}^{i} (-1)^{j} x_0 \cdots \widehat{x}_j \cdots x_i a_i \cdots a_{n-1} + \sum_{j = i}^{n-1} (-1)^{j+1} x_0 \cdots x_i a_i \cdots \widehat{a}_{j} \cdots a_{n-1}) \\
    & = \sum_{i = 0}^{n-1} \sum_{j = 0}^{i} (-1)^{i+j} x_0 \cdots \widehat{x}_j \cdots x_i a_i \cdots a_{n-1} + \sum_{i = 0}^{n-1} \sum_{j = i}^{n-1} (-1)^{i+j+1} x_0 \cdots x_i a_i \cdots \widehat{a}_{j} \cdots a_{n-1} \\
\end{align*}

At this point, we may note that for each $1 \leq i \leq n - 1$, the $(i, i)$-term of the left summation is $x_0 \cdots x_{i-1} a_i \cdots a_{n-1}$, while the $(i-1,i-1)$-term of the right summation is $-x_0 \cdots x_{i-1} a_i \cdots a_{n-1}$. Thus these terms cancel. Furthermore, the $(0,0)$-term of the left summation is $a_0 \cdots a_{n-1}=\pi x$,  and the  $(n-1,n-1)$-term of the right summation is equal to $-x_0 \cdots x_{n-1} = -x$.

For the sake of readability, we will set these terms aside and show that the remaining part of the sum is equal to $-L^0(\bd x)$. We are thus left to consider:
\[
\sum_{i = 1}^{n-1} \sum_{j = 0}^{i-1}(-1)^{i+j} x_0 \cdots \widehat{x}_j \cdots x_i a_i \cdots a_{n-1} + \sum_{i = 0}^{n-2} \sum_{j = i+1}^{n-1} (-1)^{i+j+1} x_0 \cdots x_i a_i \cdots \widehat{a}_{j} \cdots a_{n-1}
\]
We reverse the order of summation in both terms, summing first over $j$, then over $i$. Thus we obtain:

\begin{align*}
& \sum_{j = 0}^{n-2} \sum_{i=j+1}^{n-1} (-1)^{i+j} x_0 \cdots \widehat{x}_j \cdots x_i a_i \cdots a_{n-1} + \sum_{j = 1}^{n-1} \sum_{i = 0}^{j-1} (-1)^{i+j+1} x_0 \cdots x_i a_i \cdots \widehat{a}_{j} \cdots a_{n-1} \\
= & \sum_{j = 0}^{n-1} \sum_{i=j+1}^{n-1} (-1)^{i+j} x_0 \cdots \widehat{x}_j \cdots x_i a_i \cdots a_{n-1} + \sum_{j = 0}^{n-1} \sum_{i = 0}^{j-1} (-1)^{i+j+1} x_0 \cdots x_i a_i \cdots \widehat{a}_{j} \cdots a_{n-1} \\
= & \sum_{j = 0}^{n-1} (-1)^j (\sum_{i=j+1}^{n-1} (-1)^{i} x_0 \cdots \widehat{x}_j \cdots x_i a_i \cdots a_{n-1} + \sum_{i = 0}^{j-1} (-1)^{i+1} x_0 \cdots x_i a_i \cdots \widehat{a}_{j} \cdots a_{n-1}) \\
\end{align*}

For each $0 \leq j \leq n - 1$, let $y^j = \bd_j x$. That is, for $0 \leq i < j$, we have $y^j_i = x_i$, while for $j \leq i \leq n - 2$ we have $y^j_i = x_{i+1}$. For each $j, i$ let $b^j_i = \pi y^j_i$. Then we may rewrite the expression above as:

\begin{align*}
& \sum_{j = 0}^{n-1} (-1)^j (\sum_{i=j+1}^{n-1} (-1)^{i} y^j_0 \cdots y^j_{i-1} b^j_{i-1} \cdots b^j_{n-2} + \sum_{i = 0}^{j-1} (-1)^{i+1} y^j_0 \cdots y^j_i b^j_i  \cdots b^j_{n-2}) \\
= & \sum_{j = 0}^{n-1} (-1)^j (\sum_{i=j}^{n-2} (-1)^{i+1} y^j_0 \cdots y^j_{i} b^j_{i} \cdots b^j_{n-2} +  \sum_{i = 0}^{j-1} (-1)^{i+1} y^j_0 \cdots y^j_i b^j_i  \cdots b^j_{n-2}) \\
= & \sum_{j = 0}^{n-1} (-1)^j \sum_{i = 0}^{n-2} (-1)^{i+1} y^j_0 \cdots y^j_i b^j_i \cdots b^j_{n-2} \\
= & \sum_{j = 0}^{n-1} (-1)^j (-L^0(y^j)) \\
= & -L^0(\sum_{j = 0}^{n-1} (-1)^j y^j) \\
= & -L^0(\bd x) \\
\end{align*}
Thus the statement is proven.
\end{proof}

\begin{lemma}\label{L-Omega-to-Omega}
For $n \geq 1$, if $p \in \tOmegaOne_{n-1}(X,A)$, then $L^0(p) \in \tOmega_{n}(X,A)$.
\end{lemma}

\begin{proof}
From \cref{L-allowed}, we have that $L^0(p) \in \tA_n(X,A)$. To see that $\bd L^0(p)$ is allowed, recall that by \cref{L-bdry}, $\bd L^0(p) = -L^0(\bd p) - p + \pi p$. We consider each of these terms individually.

\begin{itemize}
    \item $\bd p$ is allowed by assumption, hence $L^0(\bd p)$ is allowed by \cref{L-allowed}.
    \item $p$ is allowed by assumption.
    \item $\pi p$ is allowed by \cref{pi-allowed}. \qedhere
\end{itemize}
\end{proof}

From \cref{L-Omega-to-Omega}, it follows that $L^0$ restricts to define a linear map $\tOmegaOne_{n-1}(X,A) \to \tOmega_n(X,A)$.

\begin{definition}\label{iso-def}
For $n \geq 0$, we define a linear map $E \colon M_n(X,A) \to \tOmega_n(X,A)$, sending an element $(p,q) \in M_n(X,A) = \tOmegaOne_{n-1}(X,A) \oplus \Omega_{n}(X-A)$ to $L^0(p) + q$.
\end{definition}

\begin{proposition}\label{iso-pres-bdry}
The maps $E \colon M_n(X,A) \to \tOmega_n(X,A)$ of \cref{iso-def} define a map of chain complexes $E \colon M(X,A) \to \tOmega(X,A)$. Moreover, these maps define a natural transformation $E \colon M \Rightarrow \tOmega$.
\end{proposition}

\begin{proof}
Consider an arbitrary element $(p,q) \in M_n(X,A) = \tOmegaOne_{n-1}(X,A) \oplus \tOmega_{n}(X-A)$. We will show that $\bd E(p,q) = E(\bd(p,q))$.

First, consider $\bd E(p,q) = \bd(L^0(p) + q)$. Applying \cref{L-bdry}, and recalling that terms corresponding to paths contained entirely in $A$ are set to zero when computing boundaries in $\tOmega(X,A)$, this is equal to $-L^0(\bd p) - p + \bd q$. Now consider $E(\bd(p,q))$; by definition, this is $E(-\bd p, \bd q - p) = -L^0(\bd p) + \bd q - p$. Thus we see that the two terms are equal.

To prove the naturality of $E$, we must show that the following square commutes, for any pushout square $f \colon (X,A) \to (Y,B)$.
\[
\begin{tikzcd}
M(X,A) \arrow[r,"E"] \arrow[d] & \tOmega(X,A) \arrow[d] \\
M(Y,B) \arrow[r,"E"] & \tOmega(Y,B)
\end{tikzcd}
\]
This follows by a straightforward computation.
\end{proof}

\subsection*{Proof of excision axiom}

Our next goal will be to prove that $E \colon M \Rightarrow \tOmega$ is a natural isomorphism, from which the proof of \cref{pushout-rel-hom} will follow.
For this, we will require some further lemmas characterizing the elements of $\tOmega_n(X,A)$. As the proofs of these lemmas are very long and technical, we will first discuss a simple example to illustrate some of the essential ideas behind them. 

Suppose that an element $\omega \in \tOmega_3(X,A)$, viewed as a sum of non-degenerate allowed paths, contains a term of the form $c x_0 x_1 a_1 a_2$, where $c$ is a non-zero element of $R$, $x_0, x_1 \in (X-A)_V$ and $a_1, a_2 \in A_V$. For brevity, we let $z$ denote the path $x_0 x_1 a_1 a_2$. 

Consider the boundary of $z$: its $2$-face, in particular, is $c x_0 x_1 a_2$. The assumption that the original path was non-degenerate implies that $a_1 \neq a_2$. Applying \cref{pi-closest}, we see that $\pi x_1 = a_1$ (because there is an edge $x_1 \to a_1$), but that because $a_1 \neq a_2$, there is no edge $x_1 \to a_2$. Thus $x_0 x_1 a_2$ is not allowed. 

By assumption, the boundary of $\omega$ is allowed, so there must be some other terms of $\bd \omega$, arising from the boundaries of other terms of $\omega$, which will cancel this one. Thus the linear combination $\omega$ must contain some sum $\sum_{i=1}^m c_i z_i$ such that we can obtain $x_0 x_1 a_2$ from each path $z_i$ by omitting a vertex. Because all terms of $\omega$ correspond to allowed paths, and there is no edge $x_1 \to a_2$, for each $z_i$ the vertex to be omitted must appear in between $x_1$ and $a_2$, \ie{} we must have $z_i = x_0 x_1 v_i a_2$ for some vertex $v_i$. Thus $x_0 x_1 a_2$ is again the $2$-face of each $z_i$, and hence appears in their boundaries with positive sign; it follows that $\sum_{i = 1}^m c_i = -c$, so that these terms, when added to $c x_0 x_1 a_2$, will give $0$. We may assume that we are working with a non-redundant presentation of $\omega$, so that the $z_i$ and $z$ are all distinct; in particular, this implies $v_i \neq a_1$ for all $i$ as $z$ and $z_i$ can differ only in this vertex. By \cref{pi-closest}, it follows that no $v_i$ is a vertex of $A$, as each one admits an edge from $x_1$, and the only vertex of $A$ admitting an edge from $x_1$ is $a_1$. Thus each $v_i$ is a vertex of $X - A$ for which there exist edges $x_1 \to v_i$ and $v_i \to a_2$. Applying \cref{pi-closest} again, it follows that $\pi v_i = a_2$ for all $i$.
If we then let $x^i_2 = v_i$ for all $i$, we can rearrange the sum $c z + \sum_{i=1}^m c_i z_i$ as follows:

\begin{align*}
c z + \sum_{i=1}^m c_i z_i & = cx_0 x_1 a_1 a_2 + \sum_{i=1}^m c_i z_i \\
& = - \sum_{i=1}^m c_i x_0 x_1 a_1 a_2 + \sum_{i=1}^m c_i x_0 x_1 x_2^i a_2 \\
& = \sum_{i=1}^m(- c_i x_0 x_1 a_1 a_2 + c_i x_0 x_1 x_2^i a_2) \\
& = \sum_{i=1}^m c_i L^1(x_0 x_1 x_2^i) \\
& = c L^1(x_0 x_1 x_2^i) \\
\end{align*}
Thus $z$ appears in the linear combination $\omega$ as part of an `$L$-term,' that is, a term in the image of some $L^j$.

In \cref{iso-surjective-forward}, we will generalize the reasoning of this example to show that the terms of any element of $\tOmega_n(X,A)$ which correspond to paths intersecting $A$ can be grouped into $L$-terms. Then, in \cref{iso-surjective}, we will apply similar reasoning to show that any such element may be expressed as a sum of terms in $\Omega_n(X-A)$ and terms in the image of $L^0$.


\begin{lemma}\label{iso-surjective-forward}
For $n \geq 0$, every element  $\omega \in \tOmega_n(X,A)$ can be written as 
\[
\omega = q + \sum_{k=1}^m c_k L^{j_k}(x^k) 
\]
for some $q \in A_n(X-A)$ and some set of indices $0 \leq k \leq m$ and $0 \leq j_k \leq n - 1$, where $x^k = x^k_0 \cdots x^k_{n-1}$ are distinct non-degenerate allowed paths satisfying $h(x^k_{n-1}) = 1$ and $c_k \in R$ are nonzero. 
\end{lemma}

\begin{proof}
We first note that any element $\omega \in \tOmega_n(X,A)$ may be expressed as
\begin{equation}\label{eq:intermediate} \omega=
q + \sum_{k=1}^m c_k L^{j_k}(x^k) + \sum_{r = 1}^{s} d_r y^r
\end{equation}
where:
\begin{itemize}
    \item $q + \sum_{k=1}^m c_k L^{j_k}(x^k)$ satisfies the conditions given in the statement;
    \item $s \geq 0$, each $d_r$ is a nonzero element of $R$ and each $y^r$ is a distinct allowed path in $X$ intersecting both $X - A$ and $A$.
\end{itemize}
To obtain an expression as in \cref{eq:intermediate} for an arbitrary element $\omega$, we group its terms which do not intersect $A$ together as $q$, and take the sum of the remaining terms to be $\sum_{r = 1}^{s} d_r y^r$, setting $m = 0$. Furthermore,  we may assume without loss of generality that our chosen presentation of $\omega$ is non-redundant, \ie{} that the terms $y^r$ and those of $q$ all represent distinct non-degenerate paths.

The form of \cref{eq:intermediate} essentially represents an intermediate state between an arbitrary expression for $\omega$ and an expression of the form given in the statement of the lemma. The sum $\sum_{k=1}^m c_k L^{j_k}(x^k)$ consists of those terms of $\omega$ which have been grouped together into $L$-terms as required by the statement, while $\sum_{r = 1}^{s} d_r y^r$ consists of those terms intersecting $A$ which remain ungrouped. Though we will proceed by a multi-stage induction over several variables, the core of our approach will be to group the terms $d_r y^r$ together into $L$-terms. 

Given an element $\omega \in \tOmega_{n-1}(X,A)$ expressed as in \cref{eq:intermediate}, we will show by induction on $s$ that $\omega$ may be expressed in the form given in the statement. The base case $s = 0$ is trivial, as this is precisely the case in which the two forms coincide.

Now let $s \geq 1$ and suppose the result is proven for all $0 \leq s' < s$. Choose an arbitrary term $d_r y^r$; for ease of notation we will rename $d_r$ to $e$ and $y^r$ to $z$. Note that because $e z$ is not a term of $q$, some vertex of $z$ must be contained in $A_V$. Because there are no arrows out of $A$, there exists $0 \leq j \leq n-1$ such that $z_i \in (X-A)_V$ for $0 \leq i \leq j$, while $z_i \in A_V$ for $j + 1 \leq i \leq n$. Write $z = z_0 \cdots z_j b_j \cdots b_{n-1}$ for $z_i \in (X-A)_V$ and $b_i \in A_V$. This path is allowed, so by \cref{pi-closest} we conclude that $h(z_j) = 1$ and $\pi z_j = b_j$.

We now show that for any $t$ with $j \leq t \leq n-1$, $\omega$ may be expressed as
\[
q + \sum_{k=1}^m c_k L^{j_k}(x^k) + \sum_{r' = 1}^{s'} d_{r'} y^{r'} + \sum_{l = 1}^{u} \sum_{i = j}^{t} (-1)^{j-i} e^l z^l_0 \cdots z^l_i b^l_i \cdots b^l_{n-1}
\]
for some $s' < s$ and some family of paths $z^l_0 \cdots z^l_t$ and coefficients $e^l \in R$ indexed by $1 \leq l \leq u$ for some $u \geq 1$, such that $h(z^l_i) = 1$ for all $i$ between $j$ and $t$ inclusive, where $b^l_i$ denotes $\pi z^l_i$ for each $(l,i)$. We assume that the terms of the double summation are distinct from those of the other summands in this presentation of $\omega$, but not necessarily from each other, i.e., we may have $z^l = z^{l'}$ for some $l, l'$. In the case $t = n-1$ the double summation will simply become a sum of $L$-terms, allowing us to apply the induction hypothesis on $s$ to conclude the overall proof.

We proceed by induction on $t$. In the base case $t = j$, we set $u = 1$, so that the given double sum is simply a single term $ e^1 z^1_0 \cdots z^l_j b^l_j \cdots b^l_{n-1}$. To express $\omega$ in this form we may separate out the chosen term $e z$, designate it as $e^1 z^1$, set $s' = s-1$, and re-index the remaining terms.

Now suppose the statement holds for some $j \leq t \leq n - 2$; we will prove it for $t + 1$. For each $l$, consider the final term of the corresponding alternating sum, $(-1)^{j-t}e^l z^l_0 \cdots z^l_t b^l_t \cdots b^l_{n-1}$. The boundary of this term contains a term of the form $(-1)^{j+1} e^l z^l_0 \cdots z^l_t b^l_{t+1} \cdots b^l_{n-1}$, obtained by omitting the $(t+1)^{\mathrm{st}}$ vertex of the path. Our assumption that the path $z^l_0 \cdots z^l_t b^l_{t+1} \cdots b^l_{n-1}$ is non-degenerate implies that $b^l_t \neq b^l_{t+1}$, and hence that $\pi z^l_t \neq b^l_{t+1}$. By \cref{pi-closest}, it follows that there is no edge from $z^l_t$ to $b^l_{t+1}$; thus this path is not allowed. 

Because the boundary of $\omega$ is allowed, this term must therefore be cancelled by some set of other terms of $\bd \omega$. In other words, the linear combination $\bd \omega$ must contain some sum $\sum_{w} r_w z^l_0 \cdots z^l_t b^l_{t+1} \cdots b^l_{n-1}$, where $\sum_w r_w = (-1)^j e^l$, with each term $r_w z^l_0 \cdots z^l_t b^l_{t+1} \cdots b^l_{n-1}$ arising as a face of a term of either $q$, $\sum_{k=1}^m c_k L^{j_k}(x^k)$, $\sum_{r' = 1}^{s'} d_{r'} y^{r'}$, or $\sum_{l=1}^u \sum_{i = j}^{t} (-1)^{j-i} e^l z^l_0 \cdots z^l_i b^l_i \cdots b^l_{n-1}$. Because the paths $z^l$ are not assumed to be distinct, if there is more than one such term we may re-express the sum indexed by $l$ as follows:
\[
\sum_{i=j}^t (-1)^{j-i}e^l z^l_0 \cdots z^l_i b^l_i \cdots b^l_{n-1} 
= -\sum_w \sum_{i=j}^t (-1)^{j-i}r_w z^l_0 \cdots z^l_i b^l_i \cdots b^l_{n-1}
\]
Thus we obtain a new expression for $\omega$ in the given form for a larger value of $u$. We may therefore assume without loss of generality that $(-1)^{j+1} e^l z^l_0 \cdots z^l_t b^l_{t+1} \cdots b^l_{n-1}$ is canceled by a single term $(-1)^{j} e^l z^l_0 \cdots z^l_t b^l_{t+1} \cdots b^l_{n-1}$ arising in one of the four ways described above.

We may note that, because there is no edge $z^l_t \to b^l_{t+1}$, the only way in which $(-1)^{j} e^l z^l_0 \cdots z^l_t b^l_{t+1} \cdots b^l_{n-1}$ can arise by omitting a vertex of an allowed path is if the vertex to be omitted appears between $z^l_t$ and $b^l_{t+1}$. Thus our given presentation of $\omega$ contains a term $(-1)^{j-t-1} e^l z^l_0 \cdots z^l_t v b^l_{t+1} \cdots b^l_{n-1}$ from which this boundary term arises. We first note that this term cannot be part of $q$, as the vertices $b^l_i$ are all contained in $A$.




Next we consider the case in which $(-1)^{j-t-1} e^l z^l_0 \cdots z^l_t v b^l_{t+1} \cdots b^l_{n-1}$ is a term of the double summation $\sum_{l = 1}^{u} \sum_{i = j}^{t} (-1)^{j-i} e^l z^l_0 \cdots z^l_i b^l_i \cdots b^l_{n-1}$, corresponding to some index $l'$ (necessarily distinct from $l$ as $e^l$ appears here with opposite sign). Then we must have $v = \pi z^l_t = b^l_t$ by construction of the double sum; it follows that $z^l = z^{l'}, e^l = -e^{l'}$, so that the two terms cancel; by removing these terms and re-indexing for a smaller value of $u$, we may assume without loss of generality that this case does not occur.

It remains to consider the cases in which $(-1)^{j-t-1} e^l z^l_0 \cdots z^l_t v b^l_{t+1} \cdots b^l_{n-1}$ is a term of  $\sum_{k=1}^m c_k L^{j_k}(x^k)$ or $\sum_{r' = 1}^{s'} d_{r'} y^{r'}$. As the terms of both of these sums are assumed to correspond to paths distinct from those of the double summation, we cannot have $v = b^l_t$ in these cases. As there is an edge $z^l_t \to v$, we must therefore have $v$ equal to some vertex $z^l_{t+1}$ of $X - A$ admitting an edge from $z^l_t$ and an edge to $b^l_{t+1}$. In particular, this implies $h(z^l_{t+1}) = 1$ and $\pi z^l_{t+1} = b^l_{t+1}$

Of these two, we first consider the case in which $(-1)^{j-t-1} e^l z^l_0 \cdots z^l_{t+1} b^l_{t+1} \cdots b^l_{n-1}$ is a term of $\sum_{r' = 1}^{s'} d_{r'} y^{r'}$. In this case we may simply group the term $(-1)^{j-t-1} e^l z^l_0 \cdots z^l_{t+1} b^l_{t+1} \cdots b^l_{n-1}$ together with the summation $\sum_{i = j}^{t} (-1)^{j-i} e^l y_0 \cdots y_i b_i \cdots b_{n-1}$ to form: 
\begin{align*}
& (-1)^{j-t-1} e^l z^l_0 \cdots z^l_{t+1} b^l_{t+1} \cdots b^l_{n-1} + \sum_{i = j}^{t} (-1)^{j-i} e^l z^l_0 \cdots z^l_i b^l_i \cdots b^l_{n-1} \\
= & \sum_{i = j}^{t+1} (-1)^{j-i} e^l z^l_0 \cdots z^l_i b^l_i \cdots b^l_{n-1}
\end{align*}
Thus we have extended the sum with index $l$ by adding a suitable $(t+1)$-term.

Finally, we consider the case in which $(-1)^{j-t-1} e^l z^l_0 \cdots z^l_{t+1} b^l_{t+1} \cdots b^l_{n-1}$ is a term of 
\[c_k L^{j_k}(x^k) = \sum\limits_{i = j_k}^{n-1}(-1)^{i} c_k  x_0 \cdots x_i a_i \cdots a_{n-1}\]
for some $k$; it is then necessarily the $(t+1)$-term of this sum, because $z^l_{t+1} \in (X - A)_V$ while $b^l_{t+1} \in A_V$. 
Note that for $i \geq j_k + 1$, the $i$-term of such a sum has as its $i^{\mathrm{th}}$ face $c_k x_0 \cdots x_{i-1} a_i \cdots a_{n-1}$, while the $(i-1)$-term has as its $i^{\mathrm{th}}$ face $-c_k x_0 \cdots x_{i-1} a_i \cdots a_{n-1}$. Thus these two faces will cancel each other. As we are seeking the term of $\omega$ whose $(t+1)^{\mathrm{st}}$ face will cancel that of the $(-1)^{j-t} e^l z^l_0 \cdots z^l_t b^l_t \cdots b^l_{n-1}$ term which we identified earlier, we may therefore assume without loss of generality that $(-1)^{j-t-1} z^l_0 \cdots z^l_{t+1} b^l_{t+1} \cdots b^l_{n-1}$ is the $j_k$-term of the sum, \ie{} that $j_k = t + 1$ and $c_k = (-1)^j e^l$. Thus we may perform the following rearrangement to remove this term from $c_k L^{t+1}(x^k)$ and group it with $\sum_{i = j}^{t} (-1)^{j-i} e^l z^l_0 \cdots z^l_i b^l_i \cdots b^l_{n-1}$:

\begin{align*}
& (-1)^j e^l L^{t+1}(x^k) + \sum_{i = j}^{t} e^l (-1)^{j-i} z^l_0 \cdots z^l_i b^l_i \cdots b^l_{n-1} \\
= & \sum\limits_{i = t+1}^{n-1} (-1)^{i+j} e^l x_0 \cdots x_i a_i \cdots a_{n-1} + \sum_{i = j}^{t} (-1)^{j-i} e^l z^l_0 \cdots z^l_i b^l_i \cdots b^l_{n-1} \\
= & \sum\limits_{i = t + 2}^{n-1} (-1)^{i+j} e^l x_0 \cdots x_i a_i \cdots a_{n-1} + (-1)^{t+1+j} e^l z^l_0 \cdots z^l_{t+1} b^l_{t+1} \cdots b^l_{n-1} + \sum_{i = j}^{t} (-1)^{j+i} e^l z^l_0 \cdots z^l_i b^l_i \cdots b^l_{n-1} \\
= & \sum\limits_{i = t + 2}^{n-1} (-1)^{i+j} e^l x_0 \cdots x_i a_i \cdots a_{n-1} + \sum_{i = j}^{t+1} (-1)^{j+i} e^l z^l_0 \cdots z^l_i b^l_i \cdots b_{n-1} \\
= & (-1)^{j} e^l L^{t+2}(x^k) + \sum_{i = j}^{t+1} (-1)^{j+i} e^l z^l_0 \cdots z^l_i b^l_i \cdots b^l_{n-1}.
\end{align*}
Therefore, in this case as well, we have extended the summation with index $l$ by adding a suitable $(t+1)$-term.

Thus $\omega$ may be expressed as a sum of the given form for $i$ ranging from $j$ to $t$ for any $j \leq t \leq n-1$. Considering this result in the case $t = n-1$, we note that $\sum_{l=1}^u \sum_{i = j}^{n-1} (-1)^{j+i} z^l_0 \cdots z^l_i b^l_i \cdots b^l_{n-1} = \sum_{l=1}^u L^j(z^l_0 \cdots z^l_{n-1})$. Moreover, if $z^l = z^{l'}$ for any distinct $l, l'$, we may at this point sum the corresponding $L$-terms to obtain a single term with coefficient $e^l + e^{l'}$; thus we may now assume that the paths $z^l$ are all distinct from each other, as well as from the $x^k$. Thus we may group this term together with the sum $\sum_{k = 1}^m c_k L^{j_k}(x^k)$. After suitable re-indexing, we thus obtain an expression for $\omega$ of the form
\[
q + \sum_{k=1}^m c_k L^{j_k}(x^k) + \sum_{r' = 1}^{s'} d_{r'} y^{r'}
\]
where $s' < s$. Applying induction on $s$, we see that 
\[
\omega = q + \sum_{k=1}^m c_k L^{j_k}(x^k)
\]
for a suitable choice of indices.
Thus any $\omega \in \tOmega_n(X,A)$ may be expressed in the form given in the statement.
\end{proof}

\begin{lemma}\label{iso-surjective}
For $n \geq 0$, every element $\omega \in \tOmega_n(X,A)$ may be written as
\[
\omega = q + L^0(p)
\]
for some $p \in \tAOne_{n-1}(X,A), q \in A_n(X-A)$.
\end{lemma}

\begin{proof}
We begin by expressing an arbitrary element $\omega$ in the form given by \cref{iso-surjective-forward}:
\[
\omega = q + \sum_{k=1}^m c_k L^{j_k}(x^k) 
\]
We will assume that each lower limit $j_k$ is minimal, meaning that there is no $j'_k < j_k$ such that $L^{j'_k}(x^k) = L^{j_k}(x^k)$. 
Let $m$ denote the number of lower limits $j_k$ which are nonzero. 
We will prove the statement by induction on $m$. 
By linearity of $L^0$, it suffices to show that $\omega$ can be expressed in the form of \cref{iso-surjective-forward} with all $j_k = 0$. 


The base case $m = 0$ is trivial.
For the inductive step, we let $m \geq 1$ and suppose the statement is proven for all $m' < m$. Choose some arbitrary $k$ for which $j_k \geq 1$, and consider the term $c_k L^{j_k}(x^k)$. For ease of notation, we let $c_k = c$, $x = x^k$ and $j = j_k$.


The $j$-term of the sum $c L^j(x) = \sum_{i = j}^{n-1} (-1)^i c x_0 \cdots x_i a_i \cdots a_{n-1}$ is $(-1)^{j} c x_0 \cdots x_j a_j \cdots a_{n-1}$. The $j$-face of this term is $c x_0 \cdots x_{j-1} a_{j} \cdots a_{n-1}$. Recall that by assumption, $\pi x_{j-1} \neq a_j$; \cref{pi-closest} thus implies that this path is not allowed. Similarly to the proof of \cref{iso-surjective-forward}, the fact that the boundary of $\omega$ is allowed implies that some set of other terms of $\bd \omega$ must combine to cancel $c x_0 \cdots x_{j-1} a_{j} \cdots a_{n-1}$. Once again, the fact that there is no edge $x_{j-1} \to a_j$ and all terms of $\omega$ are allowed implies that each of these terms is of the form $r_w x_0 \cdots x_{j-1} v_w a_{j} \cdots a_{n-1}$ for some vertex $v_w$ of $X$, and $\sum_w r_w = -c$. Likewise, we may break up the term $c L^j(x)$ as a sum $-\sum_w r_w L^j(x)$ and consider each of these $L$-terms individually. 
We will also assume, without loss of generality, that no subset of the $r_w$ adds to 0 -- otherwise we may simply consider the complement of this subset to obtain a smaller set of cancelling terms.

First we consider the case in which $v_w$ is a vertex of $A$; then \cref{pi-closest} implies that $\pi x_{j-1} = v_w$, so let $a_{j-1} = v_w$. Then $r_w x_0 \cdots x_{j-1} a_{j-1} a_{j} \cdots a_{n-1}$ is necessarily the $(j-1)$-term of $c_{k'}L^{j_{k'}}(x^{k'})$. We note that the $j$-term of $L^{j_{k'}}(x^{k'})$ is $(-1)^{j}x_0 \cdots x_{j-1} x'_{j} a_{j} \cdots a_{n-1}$, for some vertex $x'_j$ of $X-A$. Thus the $j$-face of this term is $x_0 \cdots x_{j-1} a_{j} \cdots a_{n-1}$, which cancels with the $j$-face of $(-1)^{j+1}x_0 \cdots x_{j-1} a_{j-1} a_{j} \cdots a_{n-1}$. As $r_w x_0 \cdots x_{j-1} v a_{j} \cdots a_{n-1}$ is meant to be the term whose $j$-face cancels that of $-r_w x_0 \cdots x_{j-1} x_j a_{j} \cdots a_{n-1}$ which we initially identified, we may therefore assume without loss of generality that $v_w$ is not a vertex of $A$.

Let us therefore assume that $v_w$ is a vertex of $X-A$, and let $x'_j = v_w$; then \cref{pi-closest} implies that $\pi x'_j = a_j$. The term $r_w x_0 \cdots x_{j-1} x'_j a_{j} \cdots a_{n-1}$ is necessarily the $j$-term of $c_{k'}L^{j_{k'}}(x^{k'})$, implying $r_w = (-1)^j c_{k'}$. 

If $j_{k'} < j$, then the $(j-1)$-term of $-L^{j_{k'}}(x^{k'})$ is $(-1)^{j}x_0 \cdots x_{j-1} a_{j-1} a_{j} \cdots a_{n-1}$. The $j$-face of this term is $x_0 \cdots x_{j-1} a_{j-1} \cdots a_{n-1}$, which cancels the $j$-face of $(-1)^{j+1}x_0 \cdots x_{j-1} x'_j a_{j} \cdots a_{n-1}$. Since the term of $\omega$ whose $j$-face cancels that of $-r_w x_0 \cdots x_{j-1} x_j a_{j} \cdots a_{n-1}$ is the term $r_w x_0 \cdots x_{j-1} x'_j a_{j} \cdots a_{n-1}$, we may assume without loss of generality that this case does not occur, \ie{} that $j_{k'} = j$.

For ease of notation, denote $x^{k'}$ by $x'$; then we have shown that $\omega$ contains terms of the form $-r_w L^j(x)$ and $r_w L^j(x')$. Moreover, by comparing the $j$-terms of the sums $L^j(x)$ and $L^j(x')$ we see that for $0 \leq i \leq j-1$ we have $x_i = x'_i$, and for $j \leq i \leq n - 1$ we have $\pi x_i = \pi x'_i = a_i$, and hence
\[r_w \sum\limits_{i=0}^{j-1} (-1)^i x_0 \cdots x_i a_i \cdots a_{n-1}
= r_w \sum\limits_{i=0}^{j-1} (-1)^i x'_0 \cdots x'_i a_i \cdots a_{n-1}\]
We may therefore add their difference to $\omega$ without changing its value. Regrouping our terms, we obtain:
\begin{align*}
& (\sum\limits_{i=0}^{j-1} (-1)^i r_w x'_0 \cdots x'_i a_i \cdots a_{n-1} + L^j(x')) - (\sum\limits_{i=0}^{j-1} (-1)^i r_w x_0 \cdots x_i a_i \cdots a_{n-1} + L^j(x)) \\
= & r_w L^0(x') - r_w L^0(x) \\
\end{align*}

Repeating this procedure for all $w$, we obtain an expression for $\omega$ as a sum of $q$ with a set of $L$-terms, in which fewer than $m$ of the lower limits of the $L$-terms are nonzero. By induction, it follows that $\omega$ can be expressed in such a form with all lower limits equal to zero.
\end{proof}
We are finally equipped to prove the following:
\begin{proposition}\label{cone-omega-iso}
The map $E \colon M(X,A) \to \tOmega(X,A)$ is an isomorphism of chain complexes.
\end{proposition}

\begin{proof}
We first show that each map $E \colon M_n(X,A) \to \tOmega_n(X,A)$ is injective. To see this, let $(p,q) \in \tOmegaOne_{n-1}(X,A) \oplus \tOmega(X,A)$ such that $q + L^0(p) = 0$. Rearranging this expression, we obtain $q = -L^0(p)$. 
Since every nonzero term of $L^0(p)$ contains a vertex of $A$ while no nonzero term of $q$ contains a vertex of $A$, it follows that both $q$ and $L^0(p)$ are zero. Thus $p = 0$ by \cref{L-injective}.

We now prove surjectivity of $E$. For some $n \geq 0$, let $\omega \in \tOmega(X,A)$. By \cref{iso-surjective},  \[\omega = L^0(p) + q\]for some $p \in \tAOne_{n-1}(X,A)$ and $ q \in A_n(X-A)$. To prove that $\omega$ is in the image of $E$ we must show that $p \in \tOmegaOne_{n-1}(X,A)$ and $q \in \Omega_n(X-A)$, \ie{} that the boundaries of both $p$ and $q$ are allowed, and that all terms of the boundary of $p$ have their last vertices in $A$.

We now proceed analogously to the proof of \cite[Lem.~3.10]{grigor'yan-jimenez-muranov-yau}. By \cref{L-bdry}
\[
\bd \omega = -L^0(\bd p) - p + \pi(p) + \bd q.
\]
Rearranging this equation, we obtain:
\[
\bd \omega + p - \pi(p) = -L^0(\bd p) + \bd q.
\]
On the left-hand side of the equation, the terms $\bd \omega$ and $p$ are allowed by assumption, while $\pi(p)$ is allowed by \cref{pi-allowed}. Thus we see that $-L^0(\bd p) + \bd q$ is allowed. Now we may note that every nonzero term of $-L^0(\bd p)$ includes at least one vertex of $A$, while this is not the case for any term of $\bd q$. 
Thus there can be no cancellations between that terms of $-L^0(\bd p)$ and $\bd q$, implying that $L^0(p)$ and $\bd q$ must each be allowed. Thus $q \in \Omega_n(X-A)$. Furthermore, by \cref{L-allowed}, the fact that $L^0(\bd p)$ is allowed implies $\bd p \in \tAOne_{n-2}(X,A)$.
\end{proof}

\begin{corollary}\label{Q-T-core}
The functors $Q, \tOmega \colon \Cof \to \Ch_R$ send all morphisms of $\CofPO$ to isomorphisms.
\end{corollary}

\begin{proof}
This is immediate from \cref{Omega-U,mapping-cone-invariant,cone-omega-iso}.
\end{proof}

We can now prove the main result of this section.

\begin{proof}[Proof of \cref{pushout-rel-hom}]
The relative homology maps $H_n(X,A) \to H_n(X',A')$ are induced by the map $\Omega(X) / \Omega(A) \to \Omega(X') / \Omega(A')$, \ie{} the image of $(X,A) \to (X',A')$ under the functor $Q \colon \Cof \to \Ch_R$. This is an isomorphism of chain complexes by \cref{Q-T-core}.
\end{proof}

\section{Main Theorem}\label{sec:main-thm}

In this section, we prove the main theorem of the paper:

\begin{theorem} \label{cofib-cat}
For any ring $R$, the category $\DiGraph$ of directed graphs admits the structure of a cofibration category, with the cofibrations as defined by \cref{cofib-def} and $R$-homology isomorphisms of \cref{def:homology-iso} as the weak equivalences.
\end{theorem}

Most of the axioms of \cref{cofib-cat-def} have already been proven in \cref{sec:cofibs,sec:excision}; those which remain to be proven are the factorization of codiagonal maps and the closure of acyclic cofibrations under transfinite composition. We first consider the factorization axiom. As a preliminary, we note the following:

\begin{proposition}\label{box-prod-weqs}
A box product of homology isomorphisms is a homology isomorphism.
\end{proposition}

\begin{proof}
This is immediate from \cite[Thm.~4.7]{grigor'yan-muranov-yau}.
\end{proof}


\begin{proposition}\label{codiag-factor}
For every $X \in \DiGraph$, the codiagonal map $X \sqcup X \to X$ factors as a cofibration followed by a path homology isomorphism.
\end{proposition}

\begin{proof}
Let $J$ denote the directed graph pictured below:

\begin{figure}[H]
    \centering
    \begin{tikzpicture}[node distance=20pt]
        \node(m2) {$\bullet$};
        \node(m1) [right=of m2] {$\bullet$};
        \node(0) [right=of m1] {$\bullet$};
        \node(1) [right=of 0] {$\bullet$};
        \node(2) [right=of 1] {$\bullet$};

        \node[above=1pt of m2] {$-2$};
        \node[above=1pt of m1] {$-1$};
        \node[above=1pt of 0] {$0$};
        \node[above=1pt of 1] {$1$};
        \node[above=1pt of 2] {$2$};              
        
        \draw[->] (m1) to (m2);
        \draw[->] (m1) to (0);
        \draw[->] (1) to (0);
        \draw[->] (1) to (2);     
    \end{tikzpicture}
\end{figure}

Let $\partial J$ denote the subcomplex of $J$ consisting of the two endpoint vertices $-2$ and $2$. The inclusion $\partial J \cto J$ is a cofibration, with projecting decomposition given by $\pi(-1) = \pi(-2) = -2, \pi(1) = \pi(2) = 2$.  Furthermore, $J$ has trivial path homology by \cref{rmk:tree_homology}.

Now observe that for any directed graph $X$, the codiagonal map $X \sqcup X \to X$ is isomorphic to the box product of the identity on $X$ with the unique map $\partial J \to I_0$. Thus the codiagonal may be factored as:
\[
X \gtimes{\partial J} \cto X \gtimes{J} \to X
\]
The map $X \gtimes{\partial J} \cto X \gtimes{J}$ is the box product of the identity on $X$ with the cofibration $\partial J \cto J$, hence a cofibration by \cref{box-prod-cofibs}. Similarly, the map $X \gtimes{J} \to X$ is the box product of the identity on $X$ with the homology isomorphism $J \to I_0$, hence a homology isomorphism by \cref{box-prod-weqs}.
\end{proof}

Next we consider the stability of acyclic cofibrations under transfinite composition.

\begin{proposition}\label{Omega-filtered-colim}
For any ring $R$, the functor $\Omega \from \DiGraph \to \Ch_R$ preserves filtered colimits.
\end{proposition}

\begin{proof}
As colimits of chain complexes are computed component-wise, it suffices to check $\Omega_n \from \DiGraph \to \ModR$ preserves filtered colimits for each $n \in \mathbb{N}$.
    
    For a graph $X$, the $R$-module $\Omega_n X$ is the pullback:
    \[ \begin{tikzcd}
        \Omega_n X \ar[r] \ar[d] \pullback & A_{n-1}(X) \ar[d, hook] \\
        A_{n}(X) \ar[r] & C_{n-1}(X).
    \end{tikzcd} \]
    As filtered colimits commute with finite limits, it suffices to show the functors
    \begin{align*}
        A_{n-1}(-) \from \DiGraph \to \ModR \\
        A_n(-) \from \DiGraph \to \ModR \\
        C_{n-1}(-) \from \DiGraph \to \ModR \\
    \end{align*}
    preserve filtered colimits.
    This follows since the sets $\{ 0, \dots, n-1 \}, \{ 0, \dots, n \}$ are finite and the graph $I_{n-1}$ is finite.
\end{proof}

\begin{corollary} \label{homology-filtered-colim}
    For any ring $R$, the path homology functor $H_* \from \DiGraph \to \ModRN$ preserves filtered colimits.
\end{corollary}
\begin{proof}
    This is immediate from \cref{Omega-filtered-colim}, together with the fact that the homology functor on chain complexes $H_* \from \Ch_R \to \ModRN$ preserves filtered colimits.
\end{proof}

\begin{proposition} \label{weq-transf-comp}
    A transfinite composite of weak equivalences is a weak equivalence.
\end{proposition}
\begin{proof}
    By definition, weak equivalences are maps which become isomorphisms under the path homology functor $H_* \from \DiGraph \to \ModRN$.
    The result then follows from \cref{homology-filtered-colim}.
\end{proof}

\begin{proof}[Proof of \cref{cofib-cat}]
We consider each of the axioms of \cref{cofib-cat-def}.
\begin{itemize}
    \item[(C1)] The class of cofibrations contains all identity maps and is closed under composition by \cref{cof-wide-subcat}. 
    The analogous results for weak equivalences are immediate from the functoriality of path homology.
    \item[(C2)] The 2-out-of-6 property for weak equivalences is immediate from the corresponding property for isomorphisms and the functoriality of path homology.
    \item[(C3)] All objects of $\DiGraph$ are cofibrant by \cref{empty-cof}.
    \item[(C4)] The existence of pushouts of cofibrations is trivial, as $\DiGraph$ is cocomplete. 
    Stability of cofibrations under pushout is given by \cref{cof-pushout}.
    Given a pushout square 
    \[
      \begin{tikzcd}
        A \arrow[r] \arrow[d,tail] & B \arrow[d,tail] \\
        X \arrow[r] & Y \\
      \end{tikzcd}
    \]
    with $A \cto X$ an acyclic cofibration, we can view it as a morphism in $\Cof$.
    It follows by \cref{cof-pushout} that $B \cto Y$ is a cofibration as well.
    By \cref{tcof-rel-hom}, each relative homology group $H_n(X,A)$ is trivial; and by \cref{pushout-rel-hom}, so is each relative homology group $H_n(Y,B)$.
    Thus $B \cto Y$ is a weak equivalence by \cref{tcof-rel-hom}.
    \item[(C5)] Factorization of codiagonal maps is given by \cref{codiag-factor}.
    \item[(C6)] The existence of small coproducts is trivial, as $\DiGraph$ is cocomplete.
    \item[(C7)] The closure of (acyclic) cofibrations under transfinite composition follows from \cref{closed-under-transf-comp,weq-transf-comp}. \qedhere
\end{itemize}
\end{proof}

Our results also allow us to compare our cofibration category structure on $\DiGraph$ with the cofibration categories of chain complexes defined in \cref{Ch-cof-cat}.

\begin{theorem}\label{Omega-exact}
For any ring $R$, the functor $\Omega \colon \DiGraph \to \Ch_R$ factors through the full subcategory of chain complexes of projective $R$-modules, and is exact when considered as a functor from the cofibration category of \cref{cofib-cat} to either $\Ch^{\proj}_R$ or $\Ch_R$.
\end{theorem}

\begin{proof}
To see that $\Omega$ factors through $\Ch^{\proj}_R$, we note that for any directed graph $X$, each abelian group $\Omega_n(X)$ is free, as a subgroup of the free abelian group $C_n(X)$.

Now we consider exactness of $\Omega$. It suffices to prove the statement for $\Ch^\proj_R$; the statement for $\Ch^\inj_R$ will then follow by \cref{Ch-inclusion-exact}. 

That $\Omega$ preserves weak equivalences is immediate, as the path homology isomorphisms of directed graphs are by definition the maps which $\Omega$ sends to quasi-isomorphisms.
To see that $\Omega$ preserves the initial object, we observe that $\Omega(\varnothing)$ is zero in each degree.

For $\Omega$ to preserve cofibrations means that it sends cofibrations of directed graphs to inclusions of chain complexes with degreewise projective cokernel. This follows from \cref{Omega-U} and the fact that for any cofibration $A \cto X$, the abelian group $\tOmega_n(X,A)$ is free as a subgroup of the free abelian group $C_n(X)$. That $\Omega$ preserves transfinite composites of cofibrations is immediate from \cref{Omega-filtered-colim}.

To see that $\Omega$ preserves pushouts of cofibrations, consider a pushout square $f \colon (X,A) \to (X',A')$. To show that $\Omega$ sends this diagram to a pushout, it suffices to show that each functor $\Omega_n$ for $n \geq 0$ sends it to a pushout of abelian groups. By \cref{Omega-U-sum,Q-T-core}, the image of this diagram under $\Omega_n$ is isomorphic to:
\[
\begin{tikzcd}
\Omega_n(A) \arrow[r] \arrow[d] & \Omega_n(A') \arrow[d] \\
\Omega_n(A) \oplus \tOmega_n(X) \arrow[r] & \Omega_n(A') \oplus \tOmega_n(X) \\
\end{tikzcd}
\]
(Note that this is isomorphic to, yet distinct from, the diagram appearing in the proof of \cref{Omega-U-sum}: in the first component of the bottom map we have replaced the isomorphism $\tOmega_n(X,A) \cong \tOmega_n(X',A')$ with the identity on $\tOmega_n(X,A)$.)

Now consider the following composite diagram:
\[
\begin{tikzcd}
0 \arrow[r] \arrow[d] & \Omega_n(G) \arrow[r] \arrow[d] & \Omega_n(G') \arrow[d] \\
\Omega_n^U(H) \arrow[r] & \Omega_n(G) \oplus \Omega_n^U(H) \arrow[r] & \Omega_n(G') \oplus \Omega_n^U(H) \\
\end{tikzcd}
\]
The left square and the composite rectangle are pushouts, as the direct sum is the coproduct in the category of abelian groups. It follows that the right square is a pushout by the two pushout lemma.
\end{proof}

It is natural to ask about the compatibility of our cofibrations with other classes of equivalences; for instance, we might ask whether they comprise part of a category structure whose weak equivalences are the homotopy equivalences of directed graphs. 
We next show that this is not the case. Preceding the proof, we recall the notions of homotopy and homotopy equivalence as in \cite{grigor'yan-lin-muranov-yau:homotopy}.

A \emph{line digraph} of size $n$ is any graph $I$ whose vertex set is $0$, $1$, \ldots, $n$ and such that for any $i = 0, 1, \ldots, n-1$, we have either $i \to i+1$ or $i+1 \to i$.
All digraphs $I_n$ of \cref{def:I_n} are examples of line digraphs, but line digraphs also include the graph $J$ used in the proof of \cref{codiag-factor} and, e.g.,
    \begin{figure}[H]
    \centering
    \begin{tikzpicture}[node distance=20pt]
        \node(0) {$\bullet$};
        \node(1) [right=of 0] {$\bullet$};
        \node(2) [right=of 1] {$\bullet$};

        \node[above=1pt of 0] {0};
        \node[above=1pt of 1] {1};
        \node[above=1pt of 2] {2};

        \draw[->] (0) to (1);
        \draw[->] (2) to (1);
    \end{tikzpicture}
    \end{figure}
Given digraph maps $f, g \colon X \to Y$, a \emph{homotopy} from $f$ to $g$, denoted $\alpha \colon f \sim g$, is a digraph map $\alpha \from X \gtimes I \to Y$ for some line digraph $I$ of size $n$ such that $\alpha(-, 0) = f$ and $\alpha(-, n) = g$.
A digraph map $f \from X \to Y$ is a \emph{homotopy equivalence} if there exists a map $g \from Y \to X$ and homotopies $\alpha \from gf \sim \id[X]$ and $\beta \from fg \sim \id[Y]$.

By \cref{rmk:tree_homology}, all line digraphs are homotopy equivalent to the point, \ie{} contractible.
On the other hand, the cycle graphs of different length are not homotopy equivalent.
Indeed, any map from a cycle of smaller size to a cycle of larger size is homotopic to a constant, while the identity map from a cycle to itself is not.

\begin{proposition}
There is no cofibration category structure on $\DiGraph$ in which the class of cofibrations includes the maps of \cref{cofib-def} and whose weak equivalences are the homotopy equivalences.
\end{proposition}

\begin{proof}
Consider the following pushout of digraphs:
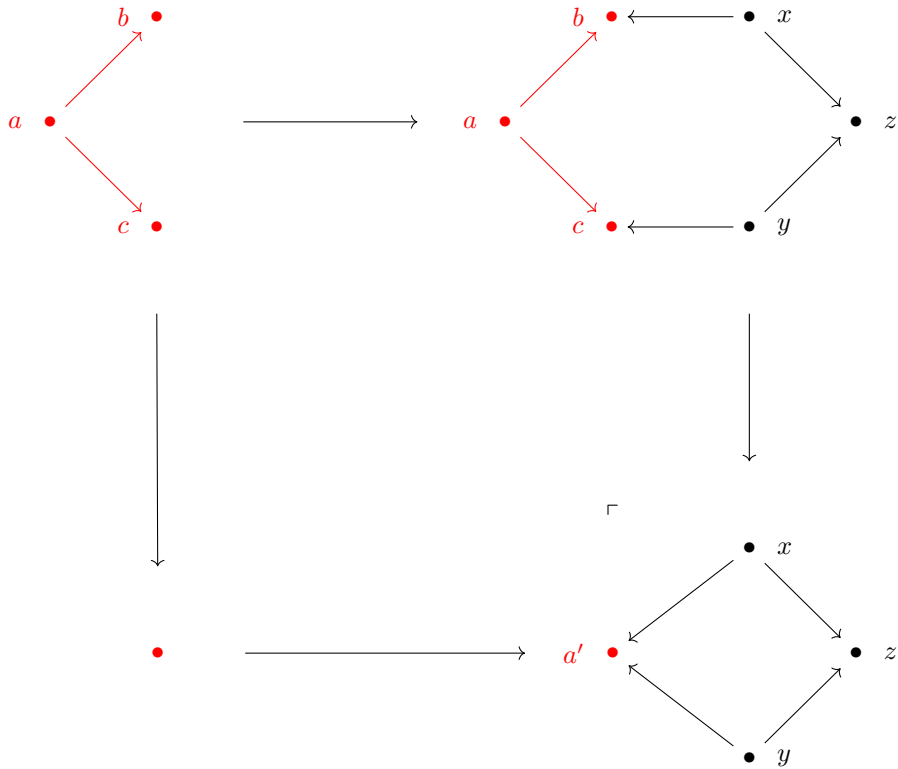
\begin{figure}[H]
    \centering
\begin{tikzpicture}[node distance = 40pt, scale=0.7]
    \node(a) [red] {$\bullet$};
    \node(b) [above right=of a, red] {$\bullet$};
    \node(x) [right=of b] {$\bullet$};
    \node(z) [below right = of x]{$\bullet$};
    \node(y) [below left = of z]{$\bullet$};
    \node(c) [left = of y, red]{$\bullet$};
    
    \node[left=1pt of a, red] {$a$};
    \node[left=1pt of b, red] {$b$};
    \node[right=1pt of x] {$x$};  
    \node[right=1pt of z] {$z$};
    \node[right=1pt of y] {$y$};
    \node[left=1pt of c, red] {$c$};  
    
    \draw[->, red] (a) to (b);
    \draw[->, red] (a) to (c);
    \draw[->] (x) to (b);
    \draw[->] (y) to (c);
    \draw[->] (x) to (z);
    \draw[->] (y) to (z);        
    
    
    \node(a2) [left=160pt of a,red] {$\bullet$};
    \node(b2) [above right=of a2, red] {$\bullet$};
    \node(c2) [below right= of a2, red]{$\bullet$};
    
    \node[left=1pt of a2, red] {$a$};
    \node[left=1pt of b2, red] {$b$};
    \node[left=1pt of c2, red] {$c$};  
    
    \draw[->, red] (a2) to (b2);
    \draw[->, red] (a2) to (c2);            
    
    
    \node(x2) [below=110pt of y] {$\bullet$};
    \node(z2) [below right = of x2]{$\bullet$};
    \node(y2) [below left = of z2]{$\bullet$};
    \node(abc) [left =80pt of z2, red]{$\bullet$};
    
    \node[right=1pt of x2] {$x$};  
    \node[right=1pt of z2] {$z$};
    \node[right=1pt of y2] {$y$};
    \node[left=1pt of abc, red] {$a'$};
    
    \draw[->] (x2) to (z2);
    \draw[->] (y2) to (z2);        
    \draw[->] (x2) to (abc);
    \draw[->] (y2) to (abc);
    
    
    \node(Dot) [left=160pt of abc, red] {$\bullet$};
    \node(Aright) [right=6em of a2] {};
    \node(C6left) [left=2em of a]{};
    \draw[->] (Aright) to (C6left);
    
    \node(Adown) [below=2em of c2] {};
    \node(Dotup) [above=2em of Dot] {};
    \draw[->] (Adown) to (Dotup);
    
    \node(C6down) [below=2em of y] {};
    \node(C4up) [above=2em of x2]{};
    \draw[->] (C6down) to (C4up);
    
    \node(Dotright) [right=2em of Dot] {};
    \node(C4left) [left=2em of abc]{};
    \draw[->] (Dotright) to (C4left);
    
    \node [above=of abc] {$\ulcorner$}; 
\end{tikzpicture}
\end{figure}
Note that the left vertical map is a homotopy equivalence and the top horizontal map is a cofibration.
If the proposed cofibration category structure were to exist, the right vertical map would be a homotopy equivalence by left properness (\cref{left-properness}).
However, since it is a map between cycles of different sizes, it cannot be a homotopy equivalence.
Hence, the proposed cofibration category structure does not exist.
\end{proof}

\bibliographystyle{amsalphaurlmod}
\bibliography{general-bibliography.bib}

\end{document}